\documentclass[10pt]{amsart}
\usepackage{mathtools}
\usepackage{kpfonts}
\usepackage{amsmath}
\usepackage{amsthm}
\usepackage{amssymb}
\usepackage{mathdots}
\usepackage{pst-node}
\usepackage{enumitem}
\usepackage{algpseudocode}
\usepackage{algorithmicx}

\newtheorem{theorem}{Theorem}[section]
\newtheorem{lemma}[theorem]{Lemma}
\newtheorem{proposition}[theorem]{Proposition}
\newtheorem{corollary}[theorem]{Corollary}
\newtheorem{example}[theorem]{Example}

\theoremstyle{remark}
\newtheorem{remark}[theorem]{Remark}

\numberwithin{equation}{section}

\def\dim{\mathop{\rm dim}\nolimits}

\def\Rea{\mathop{\rm Re}\nolimits}
\def\Ima{\mathop{\rm Im}\nolimits}

\def\Skew{\mathop{\rm Skew}\nolimits}
\def\O{\mathop{\rm O}\nolimits}
\def\SO{\mathop{\rm SO}\nolimits}
\def\GL{\mathop{\rm GL}\nolimits}
\def\U{\mathop{\rm U}\nolimits}

\parskip=\smallskipamount

\makeatletter
\renewcommand*\env@matrix[1][*\c@MaxMatrixCols c]{%
  \hskip -\arraycolsep
  \let\@ifnextchar\new@ifnextchar
  \array{#1}}
\makeatother

\begin{document}

\title[Isotropy groups]{Isotropy groups of the action of orthogonal similarity on skew-symmetric and on complex orthogonal matrices}
\author{Tadej Star\v{c}i\v{c}}
\address{Faculty of Education, University of Ljubljana, Kardeljeva Plo\v{s}\v{c}ad 16, 1000 Lju\-blja\-na, Slovenia}
\address{Institute of Mathematics, Physics and Mechanics, Jadranska
  19, 1000 Ljubljana, Slovenia}
\email{tadej.starcic@pef.uni-lj.si}
\date{December 16, 2025}

%\dedicatory{}

\keywords{
isotropy groups, complex orthogonal matrix, skew-symmetric matrix, matrix equation, Toeplitz matrix\\
\indent Research supported by grants P1-0291 
from ARRS, Republic of Slovenia.}

\begin{abstract} 
We compute and analyze isotropy subgroups of the complex orthogonal group with respect to the similarity transformation on itself and on skew-symmetric matrices. Their group structure is related to a  
%matrix
group of certain nonsingular block matrices whose blocks are rectangular block Toeplitz.  
%A key ingredient in our proof is an algorithm giving solutions of a \footnote{certain} rectangular block (complex-alter\-na\-ting) upper triangular Toeplitz matrix equation. 
\end{abstract}

\maketitle

\section{Introduction}

Let 
$\O_n(\mathbb{C})\subset \GL_n(\mathbb{C})$ be the subgroup of \emph{orthogonal}
matrices in the group of com\-plex nonsingular $n$-by-$n$ matrices $\GL_n(\mathbb{C})$;
$Q$ is orthogonal precisely when $Q^{-1}=Q^{T}$. By $\SO_n(\mathbb{C})$ we further denote its so-called special orthogonal subgroup of all $Q\in \O_n(\mathbb{C})$ with $\det Q=1$, and $\O_n(\mathbb{R})$ is the real orthogonal subgroup. Note that complex orthogonal group is much different 
from the unitary group $\U(n)\subset \GL_n(\mathbb{C})$ consisting of all $U$ with $U^{-1}=\overline{U}^{T}$.
Also, let $\Skew_n (\mathbb{C})$ be the 
%complex 
vector space of $n$-by-$n$ \emph{skew-symmetric} matrices; $A$ is skew-symmetric if and only if $A=-A^{T}$.

We consider the action of \emph{orthogonal similarity} (coinciding with \emph{orthogonal congruence})
defined on $\Skew_n(\mathbb{C})$ and on $\O_n(\mathbb{C})$, acting by $(Q,M)\mapsto Q^{-1}MQ=Q^{T}MQ$:
\begin{align}\label{aos} 
&\Phi_1\colon \O_n(\mathbb{C})\times \Skew_n(\mathbb{C})\to \Skew_n(\mathbb{C}),\\%\qquad (Q,A)\mapsto Q^{T}AQ,\\
&\Phi_2\colon \O_n(\mathbb{C})\times \O_n(\mathbb{C})\to \O_n(\mathbb{C}),%\qquad (Q,R)\mapsto Q^{-1}RQ.
\label{aoso}
\end{align}
More generally, the congruence transformations play an important role in the study of the geometry of skew-symmetric matrices by the result of Hua \cite[Theorem 10]{Hua45} 
%and Liu \cite
%[Theorem 6.4]
%{Liu} 
(cf. \cite
%[Theorem 4.4]
{Wan}). 
Also,
(\ref{aos}) is a representation of $\O_n(\mathbb{C})$ as a classical group on a vector space $\Skew_n(\mathbb{C})$ (see a monograph by Weyl \cite{Weyl}), while (\ref{aoso}) is a complex Lie group action of $\O_n(\mathbb{C})$ on itself by conjugation.

Important pieces of information about any action are the \emph{isotropy groups} (see a textbook 
%by Milne 
\cite{Milne} for their basic properties).
%In the case of 
We denote the isotropy group at an $n$-by-$n$ skew-symmetric or orthogonal matrix $M$ (with respect to (\ref{aos}) or (\ref{aoso})) by:
% respectively, by
%
\begin{align*}%\label{isog}
\Sigma_{M}:=\{Q\in \O_n(\mathbb{C})\mid Q^{T}M Q=M\}=\{Q\in \O_n(\mathbb{C})\mid Q^{T}M Q=M\}.
%&\Sigma_{\mathcal{K}}:=\{Q\in \O_n(\mathbb{C})\mid Q^{T}\mathcal{K}Q=\mathcal{K}\}, \qquad \mathcal{K}\in \Skew_n(\mathbb{C}),\\
%&\Theta_{\mathcal{O}}:=\{Q\in \O_n(\mathbb{C})\mid Q^{T}RQ=R\}, \qquad \mathcal{O}\in \O_n(\mathbb{C}).
\end{align*}
The aim of this paper is to describe the group structure of $\Sigma_{M}$ 
% $\Sigma_\mathcal{K}$ and $\Theta_\mathcal{O}$ 
and to find its dimensions. The later also yields the codimensions of the corresponding orthogonal similarity orbits; 
for more on orbit stratification under similarity or 
congruence and related topics see e.g. \cite{EEK}, \cite{S3}, \cite{TeranDopi2}.

In the generic case, when all eigenvalues of $M$ 
%$A$ and $R$ 
are simple, it is relatively easy to 
%compute 
see that $\Sigma_{M}$ 
%$\Sigma_A$, $\Theta_{R}$ 
%and to see that it 
is 
%they are both 
equal to $\bigoplus_{j=1}^{n/2}\SO_2(\mathbb{C})$ for $n$ even and to $\bigoplus_{j=1}^{(n-1)/2}\SO_2(\mathbb{C})\oplus \pm 1$ for $n$ odd (see Proposition \ref{stabs11}). However, the situation is much more delicate in the nongeneric case, where isotopy groups are conjugate to a group of certain nonsingular block matrices whose blocks are further rectangular block Toeplitz (Theorem \ref{stabw}). For instance, it is already an interesting problem to obtain the isotropy group for 
$M=\bigoplus _{j=1}^{m}
\begin{bsmallmatrix}
0 & 1+i & 0\\
-1-i & 0 & -1+i\\
0 & 1-i & 0
\end{bsmallmatrix}
$. 
Although the general approach here seems related to that in our papers \cite{TSOC} and \cite{TSOS}, 
in which the iso\-tro\-py gro\-ups under ortho\-gonal *congruence on Hermitian and orthogonal similarity 
on symmetric matrices were studied, respectively, the\-re are a few quite subtle differences between the problem considered in this paper 
and our previous work. They give rise to certain new phenomena and also make the ana\-ly\-sis
substantially more difficult.

We add that the isotropy groups of real skew-symmetric and real orthogonal matrices
with respect to real orthogonal similarity are easier to compute; they are isomorphic to direct sums of the form $\bigoplus_{j=1}^{N}\U (n_j)\oplus \O_m(\mathbb{R})$ (Corollary \ref{stabr}).

To some extent isotropy groups of (\ref{aos}) could help us solve the problem of simultaneous reduction under congruence for a triple $(A,B,C)$ with $A,B$ skew-symmetric and $C$ nonsingular symmetric. By applying the Autonne-Takagi factorization and using a suitable orthogonal congruence transformation we 
first obtain $(A',K,I)$ with the identity $I$ and the skew-symmetric normal form $K$. 
Next, $A'$ 
%(skew-symmetric) 
is simplified by using the isotropy group of $K$. This reduction might then be applied to solve a system of transpose-Sylvester matrix equations $AX+X^TA=C$, $BX+X^TB=D$ ($D$ symmetric); for more on these equations see e.g. \cite{DK}, \cite{DKS}.

%
%
%
%
%
%IG
\section{The main results}\label{secIG}

We recall the well-known canonical forms for skew-symmetric and for orthogonal matrices under orthogonal similarity; see e.g. \cite[Ch. XI]{Gant} and \cite{Well}. Note that two skew-symmetric or two orthogonal matrices are orthogonally similar if and only if they are similar. 
The Jordan canonical form of a skew-symmetric matrix $A$ has a special form; it can only contain blocks of the form $J_m(0)$ with $m$ odd, and pairs of blocks of the form $J_m(\lambda)\oplus J_m(-\lambda)$,
in which either $\lambda$ is nonzero (and complex) or $\lambda$ vanishes and $m$ is even:
\small
\begin{align}\label{JFs}
%&\mathcal{J}(A)=\bigoplus_j J_{n_j}(\lambda_j), \qquad \lambda_j\in \mathbb{C}, \\
%&
%\hspace{-12mm}
\mathcal{J}(A)=\bigoplus_j \big( J_{\alpha_j}(\lambda_j)\oplus J_{\alpha_j}(-\lambda_j)\big)\oplus \bigoplus_k \big(J_{2\beta_k}(0)\oplus J_{2\beta_k}(0)\big)\oplus\bigoplus_l J_{2\gamma_l-1}(0),\quad \lambda_j\in \mathbb{C}^{*},
%\mathcal{J}(A)=\left(\bigoplus_j J_{\alpha_j}(\lambda_j)\right)\bigoplus \left(\bigoplus_j -J_{\alpha_j}(\lambda_j)\right)\bigoplus %\left(\bigoplus_k J_{2\beta_k-1}(0)   \right), \qquad \lambda_j\in %\mathbb{C},
\end{align}
\normalsize
in which the basic Jordan block with the eigenvalue $\lambda$ is denoted by
\begin{align}\label{Jblock}
  J_m(\lambda):=\begin{bsmallmatrix}
                                                      \lambda    &  1       & \;     & 0    \\
						      \;     & \lambda     & \ddots & \;    \\     
						      \;     & \;      & \ddots &  1     \\
                                                      0     & \;      & \;     & \lambda   
                                   \end{bsmallmatrix},\quad																	
																	\lambda\in \mathbb{C}\qquad (m\textrm{-by-}m).
\end{align}
The corresponding skew-symmetric canonical form for $A$ under similarity is then
\begin{equation}\label{NF1s}
\mathcal{K}(A)=\bigoplus_{j=1}^{} 
K_{\alpha_j}(\lambda_j)
\oplus
\bigoplus_{k=1}^{} 
K_{2\beta_k}(0)
\oplus
\bigoplus_l L_{2\gamma_l-1},
\end{equation}
where we denoted
\begin{align}
\label{Km}
&K_{m}(\lambda):=
\begin{bmatrix}
M_{m} & N_{m}(\lambda)\\
-N_{m}(\lambda) & -M_{m}
\end{bmatrix},\qquad \lambda\in \mathbb{C}\\
%
%\label{Sm}
&M_m:=
\frac{1}{2}
\begin{bsmallmatrix}
0  & 1 &             & 0 \\
-1 & \ddots &   \ddots        &   \\
   & \ddots  &  \ddots & 1 \\
0   &          & -1    & 0 \\
\end{bsmallmatrix}
,\qquad
N_m(\lambda):=
\frac{i}{2}
\begin{bsmallmatrix}
0  &                &    1  & 2\lambda\\
 &   \iddots           &   \iddots     & 1 \\
1   &  \iddots  & \iddots &  \\
2\lambda   & 1          &     & 0 \\
\end{bsmallmatrix},\qquad (m\textrm{-by-}m),\nonumber
\end{align}
\begin{equation}\label{Lm}
L_{2m-1}:=
\frac{1}{2}
\left(
\begin{bsmallmatrix}
0  & 1 &             & 0 \\
-1 & \ddots &   \ddots        &   \\
   & \ddots  &  \ddots & -1 \\
0   &          & 1    & 0 \\
\end{bsmallmatrix}
+i
\begin{bsmallmatrix}
0  &                &    1  & 0\\
 &   \iddots           &   \iddots     & 1 \\
-1   &  \iddots  & \iddots &  \\
0   & -1          &     & 0 \\
\end{bsmallmatrix}
\right), \qquad \big((2m-1)\textrm{-by-}(2m-1)\big);
\end{equation}
the first $m-1$ entries on the first upper (lower) diagonal of $L_{2m-1}$ in (\ref{Lm}) are equal to plus (minus) one half and the following $m-1$ entries are minus (plus) one half, while the first $m-1$ entries on the first upper (lower) anti-diagonal are equal to plus (minus) $\frac{i}{2}$ and the following $m-1$ entries are plus (minus) $\frac{i}{2}$.
%This canonical form is uniquely determined up to a permutation of its direct summands; 
For a tridiagonal skew-symmetric normal form check \cite{Djok2}.

Next, let $Q$ be a complex orthogonal matrix. Its Jordan canonical form can only contain blocks $J_m(\pm 1)$ with $m$ odd or pairs of block $J_{m}(\lambda)\oplus J_{m}(\lambda^{-1})$, in which either $\lambda\in \mathbb{C}\setminus \{0,\pm 1\}$ or $\lambda=\pm 1$ and $m$ is even:
\small
\begin{align}\label{JFo}
\mathcal{J}(Q)
= & \bigoplus_j \big( J_{\alpha_j}(\lambda_j)\oplus J_{\alpha_j}(\lambda_j^{-1})\big)
\oplus \bigoplus_k \big(J_{2\beta_k}(1)\oplus J_{2\beta_k}(1)\big)
\oplus\bigoplus_l J_{2\gamma_l-1}(1)\\
& \oplus \bigoplus_m \big(J_{2\delta_m}(-1)\oplus J_{2\delta_m}(-1)\big)
  \oplus\bigoplus_n J_{2\epsilon_n-1}(-1),\qquad \lambda_j\in \mathbb{C}\setminus\{0,\pm 1\}.\nonumber
\end{align}
\normalsize
Its correspponding orthogonal similarity canonical form is then the following:
\small
\begin{align}\label{NF2o}
\mathcal{O}(Q)
= 
%& 
\bigoplus_j  e^{K_{\alpha_j}(\varphi_j)}
\oplus \bigoplus_k e^{K_{2\beta_k}(0)}
\oplus\bigoplus_l e^{L_{2\gamma_l-1}}
%\\
%&
 \oplus \bigoplus_m -e^{K_{2\delta_m}(0)}
\oplus\bigoplus_n  - e^{L_{2\epsilon_n-1}},
%\qquad \lambda_j=e^{\mu_j}\in \mathbb{C}\setminus\{0,\pm 1\};\nonumber
\end{align}
\normalsize
where $\lambda_j=e^{\varphi_j}$. We choose $-e^{K_{\delta_j}(0)}$ instead of $e^{K_{\delta_j}(i\pi)}$, since it is more suitable for our applications.

By transforming the equation $Q^{-1}MQ=M$ 
%with 
% $M$ given and 
%$Q$ orthogonal 
into a simple Sylvester equation $\mathcal{J}(M)X=X\mathcal{J}(M)$ with $X=SQS^{-1}$ for a transition matrix $S$, it 
is straightforward to obtain the following fact 
(cf. Proposition \ref{resAoXXA}):

\begin{proposition} \label{stabs11}
\begin{enumerate}[label= \arabic*.,ref=\arabic*,%itemsep=1pt,
%itemindent=0pt,
leftmargin=20pt]
\item \label{stabs11a}
For $\pm\lambda_1,\ldots,\pm\lambda_m\in \mathbb{C}$ pairwise distinct let 
$\mathcal{K}^{}=\bigoplus_{j=1}^{m}\mathcal{K}_j^{}$ be of the form (\ref{NF1s}),
in which $\mathcal{K}_j^{}$ is a direct sum whose summands correspond to the eigenvalue $\lambda_j$ of $\mathcal{K}^{}$. It then follows that $\Sigma_{\mathcal{K}^{}}=\bigoplus_{j=1}^{m}\Sigma_{\mathcal{K}^{}_j}$. Furthermore, if
$\mathcal{K}_j{}= 
\begin{bsmallmatrix}
0 & i\lambda_j \\
-i\lambda_j & 0
\end{bsmallmatrix}
$ with $\lambda_j \neq 0$, then $\Sigma_{\mathcal{K}^{}_j}=\SO_2(\mathbb{C})$.

\item \label{stabs11ao}
Let $\lambda_1,\ldots,\lambda_m\in \mathbb{C}$ be pairwise distinct eigenvalues of 
$\mathcal{O}^{}=\bigoplus_{j=1}^{m}\mathcal{O}_j^{}$, which is of the form (\ref{NF2o}) and such that $\mathcal{O}_j^{}$ is a direct sum whose summands correspond to the eigenvalue $\lambda_j$. 
Then $\Sigma_{\mathcal{O}^{}}=\bigoplus_{j=1}^{m}\Sigma_{\mathcal{O}^{}_j}$. 
Moreover, when
$\mathcal{O}_j{}= 
\begin{bsmallmatrix}
\cos (\varphi_j) & \sin (\varphi_j) \\
-\sin (\varphi_j) & \cos (\varphi_j)
\end{bsmallmatrix}
$ with $\lambda_j =e^{i\varphi_j}$ such that $\varphi_j \in \mathbb{C}\setminus \{k\pi\}_{k\in \mathbb{Z}}$, we have $\Sigma_{\mathcal{O}^{}_j}=\SO_2(\mathbb{C})$.

\end{enumerate}

\end{proposition}

The generic skew-symmetric and generic orthogonal matrices (i.e. with only simple eigenvalues) are thus easy to handle.

To present a detailed description of isotropy groups in the nongeneric case, we use matrices of a special block structure. Given $\alpha=(\alpha_1,\ldots,\alpha_N)$ with 
$\alpha_{1}>\ldots >\alpha_{N}$ 
and $\mu=(m_1,\ldots,m_N)$, 
let $\mathbb{T}^{\alpha,\mu}$ be formed of all nonsingular
$N$-by-$N$ block matrices with \emph{rectangular upper triangular Toeplitz} blocks of size $\alpha_r$-by-$\alpha_s$:
\begin{align}\label{0T0}
\mathcal{X}=[\mathcal{X}_{rs}]_{r,s=1}^{N}\in \mathbb{T}^{\alpha,\mu},\quad
&\mathcal{X}_{rs}=%\Omega_r^{T}Y_{rs}\Omega_s=
\left\{
\begin{array}{ll}
\hspace{-2mm}[0\quad \mathcal{T}_{rs}], & \hspace{-1mm} \alpha_r<\alpha_s\\
\hspace{-2mm}\begin{bmatrix}
\mathcal{T}_{rs}\\
0
\end{bmatrix}, & \hspace{-1mm} \alpha_r>\alpha_s\\
\hspace{-2mm}\mathcal{T}_{rs},& \hspace{-1mm} \alpha_r=\alpha_s
\end{array}\right. \hspace{-1mm}, \quad
\mathcal{T}_{rs}=T(A_0^{rs},\ldots,A_{b_{rs}-1}^{rs}),
\end{align}
where 
$b_{rs}:=\min\{\alpha_s,\alpha_r\}$, $A_0^{rr}\in \GL_{m_r}(\mathbb{C})$, $A_j^{rs}\in \mathbb{C}^{m_r\times m_s}$. 
By $ \mathbb{C}^{m\times n}$ we denote the set of $m$-by-$n$  complex matrices, and 
let a block \emph{upper triangular Toeplitz} matrix be
\begin{align}\label{UTT}
T(A_0,\ldots,A_{n-1}):=\begin{bmatrix}
  A_{0} & A_{1}                       & \ldots &    A_{n-1}  \\
0       & \ddots   &  \ddots    & \vdots \\
% \vdots & \ddots            & A_0              & \ddots &   \vdots \\ 
% \vdots & \ddots & \ddots   & \ddots              &  \vdots\\
 \vdots &   \ddots           & \ddots   &   A_1 \\
 % \vdots &  \vdots      &  \vdots            &         &       &  A_1 &   A_{1}    \\
0              & \ldots  &  0      & A_0
\end{bmatrix},
\end{align}
where $A_0 ,\ldots,A_{n-1}$ are of the same size, $T(A_0,\ldots,A_{n-1})=[T_{jk}]_{j,k=1}^{n}$ with $T_{jk}=0$ for $j>k$ and $T_{(j+1)(k+1)}=T_{jk}$. When $A_0$ is the identity matrix, (\ref{UTT}) is upper \emph{unitriangu\-lar} Toeplitz. For example, a matrix of the form (\ref{0T0}) for $N=\alpha_2=2$, $\alpha_1=3$ is
\small
\[
\begin{bmatrix}[ccc|cc]
A & B & C &  F  &  G \\ % &  J_1\\
0   & A & B     &  0  &   F \\% &  0\\
0   & 0   & A     &  0 & 0  \\ %& 0 \\
\hline
0     & H & J                         &  D  &  E \\  % &  J_2\\
0    & 0   & H              &  0    &   D    %&  0\\
\end{bmatrix}.
\]
\normalsize
The group structure of these matrices is the following.

\begin{proposition}\label{lemanilpo}
(\cite[Lemma 2.2]{TSOS})
The set $\mathbb{T}^{\alpha,\mu}$ defined by (\ref{0T0}) is a group.
Furthermore, $\mathbb{T}^{\alpha,\mu}=\mathbb{D}\ltimes \mathbb{U}$ is a semidirect product of subgroups, where $\mathbb{D}\subset \mathbb{T}^{\alpha,\mu}$ contains all nonsingular block diagonal matrices, and $\mathbb{U}\subset \mathbb{T}^{\alpha,\mu}$ is a unipotent normal subgroup consisting of all matrices having upper unitriangular Toeplitz diagonal blocks.
\end{proposition}

Moreover, certain additional conditions on $\mathcal{X}\in \mathbb{T}^{\alpha,\mu}$ in (\ref{0T0}) are needed:
\begin{enumerate}[label={\bf (\Roman*)},ref={\Roman*},wide=0pt,itemsep=4pt]
\item \label{stabs0} 
The nonzero entries of $\mathcal{X}_{rs}$ with $r>s$ can be taken freely. 
Next, denoting $A_{0}^{rr}:=(\mathcal{X}_{rr})_{11}$, 
then $B_{r}=(A_{0}^{rr})^{T}B_{r}A_{0}^{rr}$ for some given nonsingular symmetric or skew-symmetric matrix $B_r$.

\item \label{stabs4} 
For 
$j\geq 1$ and 
any $r$
%$r\in \{1,\ldots,N\}$ with 
%$\alpha_r\geq 2$, $j\in \{1,\ldots,\alpha_r-1\}$ 
we have 
$(\mathcal{X}_{rr})_{1(1+j)}=:A_{j}^{rr}=A_{0}^{rr} B_{r}^{-1}Z_{j}^{r}+D_{j}^{r}$, in which 
$A_{0}^{rr},B_{r}$ are as in (\ref{stabs0}),
$Z_{j}^{r}=(-1)^{j+1}(Z_j^{r})^{T}$
is chosen arbitrarily, $D_j^{r}$ depends polynomially on $A_{j'}^{r'r'}$ with $j'\in \{0,\ldots,j-1\}$, $r'\in \{1,\ldots,r\}$ and on the entries of $\mathcal{X}_{rs}$ for $r>s$. 

\item \label{stabs2} 
The nonzero entries of $\mathcal{X}_{rs}$ for $r,s\in \{1,\ldots,N\}$ with $r<s$ are uniquely determined (polynomially) by the entries of $\mathcal{X}_{rs}$ with $r\geq s$ (described above).
\end{enumerate}
%
%\noindent
Given a symmetric or a skew-symmetric matrix $B$, a matrix $Q$ is called $B$-\emph{pseudo-orthogonal} if $B=Q^{T}BQ$; for instance $A_0^{rr}$ in (\ref{stabs0}) is $B_r$-pseudo-orthogonal.

The following significant examples satisfy the above conditions. By $I_m$ we denote the $m$-by-$m$ identity matrix.

\begin{example} (cf. \cite[Example 3.1]{TSOS})
For $\alpha_r$ odd (even) fix $B_{r}$
nonsingular symmetric (skew-symmetric), and for $j$ odd (even) let 
$Z_{j}^{r}$ be any symmetric (skew-symmetric) matrix, all of size $m_r\times m_r$.
Set
\small
\begin{align}\label{asZ}
&\mathcal{W}=\bigoplus_{r=1}^{N}T(I_{m_r},W_1^{r},\ldots,W_{\alpha_r-1}^{r}),\qquad
W_{j}^{r}:=B_{r}^{-1}\big(Z_{j}^{r}-\frac{1}{2}\sum_{k=1}^{j-1}(-1)^k(W_k^{r})^{T}B_{r}W_{j-k}^{r}
\big).
\end{align}
\end{example}
\normalsize

\begin{example}
(cf. \cite[Example 3.2]{TSOS}) For $1\leq p< t\leq N$ let $\alpha_t<\alpha_p$ and assume $B_{r}\in \mathbb{C}^{m_r\times m_r}$ with $r\in \{p,t\}$ is nonsingular symmetric (skew-symmetric) for $\alpha_r$ odd (even), while $F \in \mathbb{C}^{m_p\times m_t}$ is an arbitrary matrix. We define a special matrix of the form (\ref{0T0}) having the identity as a principal subma\-trix formed by all blocks except those at the $p$-th and the $t$-th columns and rows:
\small
\begin{align}
\label{Hptk}
&\qquad \quad \mathcal{T}_{rs}=\left\{
\begin{array}{ll}
\hspace{-1mm}\bigoplus_{j=1}^{\alpha_r}I_{m_r}, &  r=s\\
\hspace{-1mm}0,                        &  r\neq s 
\end{array}
\right. \hspace{-1mm}, \quad \{r,s\}\not\subset\{p,t\},\qquad\,
\mathcal{T}_{tp}=
N_{\alpha_t}^{k}(F ),\quad 0\leq k \leq \alpha_t-1,\\ 
&\quad\qquad 
\mathcal{T}_{rr}=T(I_{m_r},V_1^{r},\ldots,V_{\alpha_r-1}^{r}), \quad r\in \{p,t\},\qquad\quad\quad
\mathcal{T}_{pt}=
N_{\alpha_t}^{k}\big(-B_{p}^{-1}F^{T}B_{p}\big),
\label{Hptk2}\\
&
V_{j}^{p}:=\left\{\begin{array}{ll}
\hspace{-1mm}a_{n-1}B_{p}^{-1}(F^{T}B_{t}FB_{p}^{-1})^{n}B_p^{}, & j=n(2k+\alpha_p-\beta)\\
\hspace{-1mm}0,                      & \textrm{otherwise}
\end{array}
\right.\hspace{-1mm},\qquad a_{n}:=-\frac{1}{2^{2n+1}(n+1)}
\binom{2n}{n},
\nonumber \\
&
V_{j}^{t}:=\left\{\begin{array}{ll}
\hspace{-1mm} a_{n-1}B_{t}^{-1}(B_{t}FB_{p}^{-1}F^{T})^{n}B_{t}, & j=n(2k+\alpha_t-\beta)\\
\hspace{-1mm} 0,                      & \textrm{otherwise}
\end{array}
\right.\hspace{-1mm};\nonumber
\end{align}
\normalsize
here $N_{\alpha}^{k}(X)$ denotes an $\alpha\times \alpha$ matrix with $X$ on the $k$-th super-diagonal (main diagonal for $k=0$) and zeros otherwise.
If $N=2$, $\alpha_1=4$, $\alpha_2=2$, $k=1$,
$B_{1}=B_{2}=I$, we get
\small
\begin{equation}\label{exXC}
\begin{bmatrix}[cccc|cc]
I_{} & 0 & 0 & -\tfrac{1}{2}F^{T}F  &  0  &  -F^{T} \\% &  f\\
0   & I_{} & 0   & 0  &  0  &   0\\%  &  0\\
0   & 0   & I_{} & 0    &  0 & 0 \\%& 0 \\
0   & 0   &  0  &  I_{}   & 0 &0 \\% & 0\\
\hline
0   & 0   & 0 & F                         &  I  & 0  \\%&  0\\
0   & 0   & 0   & 0              &  0    &   I  \\%&  0\\
%\hline
%0   & 0   & 0   & f                        &  0    &   0    & 1 
\end{bmatrix}.
\end{equation}
\normalsize
\end{example}
\normalsize

We state our main result; we prove it in Sec.\ref{sec2}. In comparison to \cite[Theorem 2.3]{TSOC} and \cite[Theorem 3.2]{TSOS}, the first part of the theorem exihibits a new phenomena in the structure of the isotropy groups, while in the second part more subtle differences occur. Hence significant modifications in the proof are required.

\begin{theorem}\label{stabw}
For 
$\alpha=(\alpha_1,\ldots,\alpha_N)$ with $\alpha_{1}>\ldots >\alpha_{N}$ and 
$\mu=(m_1,\ldots,m_N)$, 
$\widetilde{\mu}=(\widetilde{m}_1,\ldots,\widetilde{m}_N)$, 
$\widetilde{m}_r:=
\left\{
\begin{array}{ll}
\hspace{-2mm}m_r, &  \hspace{-1mm}\alpha_r \textrm{ odd}\\
\hspace{-2mm}2m_r, & \hspace{-1mm} \alpha_r \textrm{ even}
\end{array}
\right.$, let $\mathbb{T}^{\alpha,\mu}$ and $\mathbb{T}^{\alpha,\widetilde{\mu}}$ be as defined by (\ref{0T0}), while $K_{\alpha_r}(\lambda)$ for $\lambda\in \mathbb{C}$ and $L_{\alpha_r}$ for $\alpha_r$ odd are defined in (\ref{Km}) and (\ref{Lm}), respectively.
\begin{enumerate}
\item \label{stabw1}
Set $\mathcal{K}^{}_{\lambda}:=\bigoplus_{r=1}^{N}\big(\bigoplus_{j=1}^{m_r}K_{\alpha_r}(\lambda)\big)$ for $\lambda\in \mathbb{C}\setminus\{0\}$ and $\mathcal{O}_{\lambda}:=e^{\mathcal{K}_{\lambda}}$ for $\lambda\in \mathbb{C}\setminus\{k\pi\}_{k\in \mathbb{Z}}$.
% $e^{\lambda}\neq \pm 1$. 
Then their isotropy groups 
$\Sigma_{\mathcal{K}^{}_{\lambda}}$, $\Sigma_{\mathcal{O}^{}_{\lambda}}$ coincide and are conjugate to
\[
\mathbb{X}:=\{\mathcal{X}\oplus (\mathcal{X}^{-1})^{T}\mid \mathcal{X}\in \mathbb{T}^{\alpha,\mu}\}\subset \mathbb{T}^{\alpha,\mu}\oplus \mathbb{T}^{\alpha,\mu}.
\]
In particular, $\Sigma_{\mathcal{K}^{}_{\lambda}}$ and $\Sigma_{\mathcal{O}^{}_{\lambda}}$ are isomorphic to $\mathbb{T}^{\alpha,\mu}$ and their dimension is 
\[
\sum_{r=1}^{N} m_r\big(\alpha_r m_r+2\sum_{s=1}^{r-1}\alpha_s m_s\big).
\] 
\item \label{stabw2}
Assume
$\mathcal{K}^{}_0=
\bigoplus_{r=1}^{N}\big(\bigoplus_{j=1}^{m_r}\widetilde{K}_{r}\big)$ with $\widetilde{K}_{r}=
\left\{
\begin{array}{ll}
\hspace{-1mm}K_{\alpha_r}(0),   &  \hspace{-1mm} 
%\textrm{all } 
\alpha_r \textrm{ even}\\
\hspace{-1mm}L_{\alpha_r},      &   \hspace{-1mm} 
%\textrm{all } 
\alpha_r \textrm{ odd}
\end{array}
\right.$ 
and  
$\mathcal{O}_0=\varepsilon e^{\mathcal{K}_0}$, $\varepsilon\in \{-1,1\}$. Then the isotropy groups 
$\Sigma_{\mathcal{K}^{}_0}$, $\Sigma_{\mathcal{O}^{}_0}$ coincide and they are conjugate to a subgroup
$\mathbb{X} \subset \mathbb{T}^{\alpha,\widetilde{\mu}}$. Furthermore, each $\mathcal{X} \in \mathbb{X}$ satisfies the conditions (\ref{stabs0}), (\ref{stabs4}), (\ref{stabs2}) for 
$B_r := \left\{
\begin{array}{ll}
\hspace{-1mm} \bigoplus_{j=1}^{m_r} 
\begin{bsmallmatrix}
0 & -1\\
1  & 0
\end{bsmallmatrix}, & \alpha_r \textrm{ even}\\
\hspace{-1mm}(-1)^{(\alpha_r-1)/2} I_{m_r}, & \alpha_r \textrm{ odd}
\end{array}
\right.$ and 
\begin{align*}
\dim_{\mathbb{C}} (\Sigma_{\mathcal{K}_{0}^{}})
=\dim_{\mathbb{C}} (\Sigma_{\mathcal{O}_{0}^{}})
=
\sum_{r=1}^{N}\big(\tfrac{1}{2}\widetilde{m}_r^{2}\alpha_r+\sum_{s=1}^{r-1}\alpha_s \widetilde{m}_r \widetilde{m}_s\big) -
\sum_{\alpha_r \textrm{ odd}}\tfrac{1}{2} m_r,
\end{align*}
\normalsize
Moreover, $\mathbb{X}=\mathbb{O}\ltimes \mathbb{V}$ is a semidirect product, in which $\mathbb{O}$ consists of all 
$\mathcal{Q}=\bigoplus_{r=1}^{N} \big(\bigoplus_{j=1}^{\alpha_r} Q_r\big)$, where $Q_r$ is $B_r$-pseudo-orthogonal, while 
$\mathbb{V}$ is generated by matrices of the form (\ref{asZ}) and (\ref{Hptk}) with (\ref{Hptk2}) for $B_{r}$ defined above.
\end{enumerate}
\end{theorem}

\begin{remark}\label{OpI}
\begin{enumerate}[label={(\roman*)},ref={\roman*},wide=0pt,itemsep=2pt]
\item The computation of $\mathbb{X}$ in Theorem \ref{stabw} (\ref{stabw2}) 
is pro\-vi\-ded as part of its proof (Lemma \ref{EqT}), as well as a matrix $\Psi$ for which $\Sigma_{\mathcal{K}_{\lambda}}=\Sigma_{\mathcal{O}_{\lambda}}=\Psi^{-1}\mathbb{X} \Psi $.
\item For any skew-symmetric or orthogonal matrix $M$, Theorem \ref{stabw} provides the dimension of the orbit $\{Q^T M Q\mid Q \textrm{ orthogonal}\}$ at $M$; 
it is equal to the codimension of $\Sigma_M$ 
%(or $\Theta_M$) 
in the group of orthogonal matrices.
\end{enumerate}
\end{remark}

Real skew-symmetric and real orthogonal matrices are normal, thus they are diagonalizable. Real skew-symmetric canonical form is a direct sum of blocks of the form $\begin{bsmallmatrix}
0 & \sigma\\
-\sigma & 0
\end{bsmallmatrix}$ with $\sigma\in\mathbb{R}^{}\setminus\{0\}$ 
and the zero matrix, while orthogonal canonical form consists of blocks 
$\begin{bsmallmatrix}
\cos (\varphi) & \sin (\varphi) \\
-\sin (\varphi) & \cos (\varphi) 
\end{bsmallmatrix}$ with $\varphi \in\mathbb{R}^{}\setminus \{k\pi\}_{k\in \mathbb{Z}}$ and the plus or minus identity matrix. Moreover, these can be obtained even by real orthogonal similarity transformations (see \cite{Thomp}). Their isotropy groups are provided by 
%somewhat technical but straightforward computations (
applying the proof of Theorem \ref{stabw} (\ref{stabw1}) for a simple special case $N=\alpha_1=1$.

\begin{corollary}\label{stabr}
\begin{enumerate}
\item 
%Let $0,\pm i\sigma_1,\ldots,\pm i\sigma_{N}$ be pairwise distinct eigenvalues of
Assume $
\mathcal{K}=\bigoplus_{j=1}^{N}\Big(\bigoplus_{k=1}^{m_j} \begin{bsmallmatrix}
0 & \sigma_j\\
-\sigma_j & 0
\end{bsmallmatrix}\Big) \oplus 0_{M}
$, in which $ \sigma_1,\ldots,\sigma_N\in \mathbb{R}^{}\setminus\{0\}$ are pairwise distinct, and set $n:=M+\sum_{j=1}^{N}2m_j$.
It then follows that $\Sigma_{\mathcal{K}^{}}\cap \O_n(\mathbb{R})$ is real conjugate (by a permutation matrix) to 
\begin{align*}%\label{rthm}
\bigoplus_{j=1}^{N}
\big\{  \begin{bsmallmatrix}
\Rea (Q) & \Ima (Q)\\
-\Ima (Q) & \Rea (Q)
\end{bsmallmatrix} \,\, \big| \,\,Q\in \U(m_j)\big\}
\oplus
\O_{M}(\mathbb{R}) \subset \mathbb{R}^{n\times n
}.
%\bigoplus_{k=1}^{n} \big( \O_{m_k}(\mathbb{R})\oplus \O_{m_k}(\mathbb{R})\big)\oplus \O_{M}(\mathbb{R}),
\end{align*}
%
%\noindent
In particular, $\Sigma_{\mathcal{K}^{}}\cap \O_n(\mathbb{R})$ is isomorphic to $\bigoplus_{j=1}^{N}
 \U(m_j)
\oplus
\O_{M}(\mathbb{R})$.

\item 
%Let $\pm 1,e^{\pm i\varphi_1},\ldots,e^{\pm i\varphi_{N}}$ be pairwise distinct eigenvalues of
Suppose 
$\mathcal{O}=\bigoplus_{j=1}^{N}\big(\bigoplus_{k=1}^{m_j} \begin{bsmallmatrix}
\cos (\varphi_j) & \sin (\varphi_j) \\
-\sin (\varphi_j) & \cos (\varphi_j) 
\end{bsmallmatrix}\big) \oplus I_{M_{1}} \oplus -I_{M_{2}}$, in which $\varphi_1,\ldots,\varphi_N\in \mathbb{R}\setminus \{k\pi\}_{k\in \mathbb{Z}}$ are paiwise distinct, and set $n:=M_{1}+M_{2}+\hspace{-1mm}\sum_{j=1}^{N}2m_j$. 
Then $\Sigma_{\mathcal{O}^{}}\cap \O_n(\mathbb{R})$ is real conjugate (by a permutation matrix) to
\begin{align*}%\label{rthm}
\bigoplus_{j=1}^{N}
\big\{  \begin{bsmallmatrix}
\Rea (Q) & \Ima (Q)\\
-\Ima (Q) & \Rea (Q)
\end{bsmallmatrix} \,\, \big| \,\,Q\in \U(m_j)\big\}
\oplus
\O_{M_{1}}(\mathbb{R})\oplus
\O_{M_{2}}(\mathbb{R}) \subset \mathbb{R}^{n\times n
}.
%\bigoplus_{k=1}^{n} \big( \O_{m_k}(\mathbb{R})\oplus \O_{m_k}(\mathbb{R})\big)\oplus \O_{M}(\mathbb{R}),
\end{align*}
In particular, $\Sigma_{\mathcal{O}^{}}$ is isomorphic to $\bigoplus_{j=1}^{N}
 \U(m_j)
\oplus
\O_{M_{1}}(\mathbb{R})\oplus \O_{M_{2}}(\mathbb{R})$.
\end{enumerate}
\end{corollary}

Since $\SO_2(\mathbb{R})$ is isomorphic to $U(1)$, Corollary \ref{stabr} 
% for all $m_j=1$ 
is consistent with Proposition \ref{stabs11}. For more on the real (imaginary) part of a unitary matrix we refer to \cite{Oli}.

\section{Certain block matrix equation}\label{cereq}

For 
$\alpha=(\alpha_1,\ldots,\alpha_N)$ with $\alpha_{1}>\ldots >\alpha_{N}$ and 
$\mu=(m_1,\ldots,m_N)$, 
let 
\small
\begin{align}\label{BBF}
&\mathcal{B}=\bigoplus_{r=1}^{N}T_a\big(B_0^{r},B_1^{r},\ldots,B_{\alpha_r-1}^{r}\big),\quad \mathcal{C}=\bigoplus_{r=1}^{N}T_a\big(C_0^{r},C_1^{r},\ldots,C_{\alpha_r-1}^{r}\big),\quad \mathcal{F}=\bigoplus_{r=1}^{N}E_{\alpha_r}(I_{m_r}),
%& B_0^{r},C_0^{r}\in GL_{m_r}(\mathbb{C})\cap \Her_{m_r} \cap \mathbb{R}^{m_r \times m_r}, \qquad B_1^{r},C_1^{r},\ldots, B_{\alpha_r-1}^{r}, C_{\alpha_r-1}^{r}\in %\Her_{m_r} \cap \mathbb{R}^{m_r \times m_r},\nonumber
\end{align}
\normalsize
in which $B_{j}^{r},C_{j}^{r}\in  \mathbb{C}^{m_r \times m_r}$ are symmetric for $\alpha_r-j$ odd and 
skew-symmetric for $\alpha_r-j$ even, $B_{0}^{r},C_{0}^{r}\in \GL_{m_r}(\mathbb{C})$, and 
$E_{\alpha}(I_{m}):=
\begin{bsmallmatrix}
 0                 &      & I_{m}\\
            &   \iddots     &  \\
%     1      &              &  \\
I_{m}            &           &  0\\
\end{bsmallmatrix}
$
is an $\alpha\textrm{-by-} \alpha$ block matrix with $I_m$ on the anti-diagonal and zero matrices otherwise; by
%
%\vspace{-1mm}
\setlength{\arraycolsep}{4.4pt}
\small
\begin{align*}%\label{Toep}
T_a(A_0,\ldots,A_{\alpha-1}):=\begin{bmatrix}
  A_{0} & A_{1}         &  \ldots             & \ldots &    A_{\alpha-1}  \\
0       & -A_0 & -A_{1}   &  \ldots    & -A_{\alpha-2} \\
 \vdots & \ddots            & A_0              & \ddots &   \vdots \\ 
 \vdots &  & \ddots   & \ddots            &  (-1)^{\alpha-2}A_1\\
 %\vdots &        &              & \ddots   &   \vdots \\
 % \vdots &  \vdots      &  \vdots            &         &       &  A_1 &       \\
0       & \ldots            &  \ldots &  0      & (-1)^{\alpha-1}A_0
\end{bmatrix}
%\qquad (\beta\textrm{-by-}\beta),
\end{align*}
\normalsize
we denote an $\alpha$-by-$\alpha$ block \emph{alternating upper triangular Toeplitz} matrix 
%(with entries of the same size) 
such that $T_a(A_0,\ldots,A_{\alpha-1})=[T_{jk}]_{j,k=1}^{\alpha}$, $T_{jk}=0$ for $j>k$, $T_{(j+1)(k+1)}=-T_{jk}$.
To prove Theorem \ref{stabw} (\ref{stabw2}) it is essential 
to find all $\mathcal{X}\in \mathbb{T}^{\alpha,\mu}$ 
%in $ \mathbb{T}^{\alpha,\mu}$ for $\alpha=(\alpha_1,\ldots,\alpha_N)$, $\mu=(m_1,\ldots,m_N)$ (see 
(see (\ref{0T0})) that solve the equation
\begin{equation}\label{eqFYFIY}
\mathcal{C}=\mathcal{F}\mathcal{X}^{T}\mathcal{F}\mathcal{B} \mathcal{X}.
\end{equation}

The following result provides the solutions of (\ref{eqFYFIY}). 
It is an adaptation of 
\cite[Lemma 4.1]{TSOS}, \cite[Lemma 4.1]{TSOC} to our situation; note that certain subtle modifications are made.
%but also some less intriguing details in the proof are omitted. 

\begin{lemma}\label{EqT}
Given 
$\alpha=(\alpha_1,\ldots,\alpha_N)$ with $\alpha_{1}>\ldots >\alpha_{N}$ and 
$\mu=(m_1,\ldots,m_N)$, 
let 
$\mathcal{B},\mathcal{C}$ be as in (\ref{BBF}). Then 
$\mathcal{X}=[\mathcal{X}_{rs}]_{r,s=1}^{N}\in \mathbb{T}^{\alpha, \mu}$ (defined by (\ref{0T0})), i.e. 
%\vspace{-1mm}
\begin{equation*}\label{EqTX}
\mathcal{X}_{rs}=
\left\{
\begin{array}{ll}
\hspace{-1mm}[0\quad \mathcal{T}_{rs}], & \alpha_r<\alpha_s\\
\hspace{-1mm}\begin{bmatrix}
\mathcal{T}_{rs}\\
0
\end{bmatrix}, & \alpha_r>\alpha_s\\
\hspace{-1mm}\mathcal{T}_{rs},& \alpha_r=\alpha_s
\end{array}\right., \qquad 
\begin{array}{l}
\\
%(\alpha_{1}>\alpha_{2}>\ldots >\alpha_{N}),\\
\mathcal{T}_{rs}=T\big(A_0^{rs},\ldots,A_{b_{rs}-1}^{rs}\big),\quad b_{rs}:=\min\{\alpha_s,\alpha_r\},\\
\\
A_j^{rs}\in \mathbb{C}^{m_r\times m_s},\quad A_0^{rr}\in GL_{m_r}(\mathbb{C}),
\end{array}
\end{equation*}
%
%(partitioned conformally to $\mathcal{B},\mathcal{C}$) 
is a solution of the equation (\ref{eqFYFIY}) if and only if it
satisfies the following properties:
\begin{enumerate}[label={(\alph*)},ref={\alph*},
%itemsep=5pt,
leftmargin=13pt,itemindent=3pt]
%\item \label{EqT4}
%\begin{enumerate}[label={(\alph*)},ref={\alph*},itemsep=5pt,leftmargin=19pt,itemindent=0pt]
%
\item \label{EqT2} 
Each $A_0^{rr}$ is a solution of the equation $C_0^{r}=(A_0^{rr})^{T}B_0^{r}A_0^{rr}$. If $N\geq 2$ matrices $A_j^{rs}$ for $j\in \{0,\ldots,\alpha_{r}-1\}$, $r,s\in \{1,\ldots,N\}$ with $r>s$ can be taken arbitrarily. 

\item \label{EqT3ca} 
Assuming (\ref{EqT2}) and choosing freely the $m_r$-by-$m_r$ matrices $Z_j^{r}$ with $(Z_j^{r})^{T}=(-1)^{j+1}Z_j^{r}$ for $j\in \{1,\ldots,\alpha_r-1\}$, the remaining entries of $\mathcal{X}$ are computed as follows:
\end{enumerate}

\vspace{-5mm}
\hspace{-12mm}
\begin{algorithmic}
\vspace{1mm}
\State $\Psi^{krs}_{n}:=\sum_{l=0}^{n}\sum_{i=0}^{n-l}(-1)^{l}(A_l^{kr})^{T}B_{n-l-i}^{k}A_i^{ks}$
\small
\State 
$\xi_n^{rs}:=
\sum_{l=1}^{n-1}\sum_{i=0}^{n-l}(-1)^{l}(A_l^{rr})^{T}B_{n-l-i}^{r}A_i^{rs}
+ \sum_{i=0}^{n-1}(A_0^{rr})^{T}B_{n-i}^{k}A_i^{rs}
%+ (A_0^{kr})^{T}B_{0}^{k}A_n^{ks}
+\left\{\begin{array}{ll}
\hspace{-2mm} 0, & s=r\\
\hspace{-2mm} (-1)^{n}(A_n^{rr})^{T}B_{0}^{k}A_0^{rs}, &  s> r
%(\overline{A}_0^{(s-p)s})^{T}A_j^{(s-p)s}, & \hspace{-2mm} b_{kr} \textrm{ %even} \\
%(A_0^{(s-p)s})^{T}A_j^{(s-p)s}, & \hspace{-2mm} \alpha_{s-p} \textrm{ even}
\end{array}
\right.
$
\normalsize
\For {$j=0:\alpha_1-1$}
    \If {$r\in \{1,\ldots,N\}$,  $ j\in \{1,\ldots,\alpha_r-1\}$}
    \State $A_j^{rr}=A_0^{rr}(C_0^{r})^{-1}\Big(Z_j^{r}-\frac{1}{2}\big(\xi_j^{rr}
    %+\sum_{l=1}^{j-1}\sum_{m=0}^{j-l}( A_{l}^{rr})^{T} B_{j-l-m}^{k}A_m^{kr}$
    %\qquad \qquad \qquad \qquad \qquad \qquad \quad $
    +\sum_{k=1}^{r-1}\Psi_{j-\alpha_k+\alpha_r}^{krr}+\sum_{k=r+1}^{N}\Psi_{j-\alpha_{r}+\alpha_k}^{krr}\big)\Big)$
    \EndIf
    \For {$p=1:N-1$}
        \If {$r\in \{1,\ldots,N\}$, $j\leq \alpha_{r+p}-1$, $r+p\leq N$}
        \small
        \State \hspace{-3mm} $A_j^{r(r+p)}=-A_0^{r(r+p)}(C_0^{r})^{-1}\Big(
        %\sum_{l=1}^{j}\sum_{m=0}^{n-l}( A_{l}^{rr})^{T} B_{j-l-m}^{r}A_m^{r(r+p)}
        \xi_j^{r(r+p)}+\sum_{k=1}^{r-1}\Psi_{j-\alpha_k+\alpha_r}^{kr(r+p)}+ \sum_{k=r+1}^{r+p}\Psi_{j}^{kr(r+p)} $\\
        \qquad \qquad \qquad \qquad \qquad \qquad \qquad \qquad \qquad \qquad \qquad \qquad $+\sum_{k=r+p+1}^{N}\Psi_{j-\alpha_{r+p}+\alpha_k}^{kr(r+p)}\Big)$   
\normalsize
        \EndIf
    \EndFor
\EndFor
\end{algorithmic}
(We define $\sum_{j=l}^{n}a_j=0$ for $l>n$, and the loop p =1 : N-1 is not performed for $N=1$.)

In particular, the 
%(complex) 
dimension of the space of solutions $\mathcal{X}\in\mathbb{T}$ of the equation (\ref{eqFYFIY}) is  
\small
\begin{align}\label{cdim}
\sum_{r=1}^{N}\big(\tfrac{1}{2}m_r^{2}\alpha_r+\sum_{s=1}^{r-1}\alpha_s m_r m_s\big) -
\sum_{\alpha_r \textrm{ odd}}\tfrac{1}{2} m_r, 
\end{align}
\normalsize

\end{lemma}

Observe that $\mathcal{X}$ in Lemma \ref{EqT} is consistent with the conditions (\ref{stabs0}), (\ref{stabs4}), (\ref{stabs2}) in Sec. \ref{secIG}. Next, the order of the calculation of the entries of $\mathcal{X}$ is the following:
\begin{itemize}[
%label={\bf (\roman*)},ref={\roman*},
wide=12pt,itemsep=2pt,leftmargin=25pt]
%\begin{itemize}
%
\item $\mathcal{X}_{rs}$ 
%: $A_j^{rs}$ 
for $r>s$ (the blocks below the main diagonal of $\mathcal{X}$; chosen freely),

\item $(\mathcal{X}_{rr})_{11}=A_0^{rr}$ (the diagonals of the main diagonal blocks of $\mathcal{X}$),
%for $r\in \{1,\ldots,N\}$. 

\item $(\mathcal{X}_{r(r+1)})_{11}=A_{0}^{r(r+1)}$\hspace{-1mm} (the diagonals of the 1st upper off-diagonal blocks of $\mathcal{X}$),

%\item $(\mathcal{X}_{r(r+2)})_{11}=A_{0}^{r(r+2)}$ (the second upper off-diagonal blocks of $\mathcal{X}$), 
%the step $j=0$, $p=3$ gives the diagonal entries of the third upper off-diagonal blocks of $\mathcal{X}$ (i.e. $(\mathcal{X}_{r(r+3)})_{11}=A_{0}^{r(r+3)}$), 

\item[\ldots]
\item $(\mathcal{X}_{r(r+p)})_{1(j+1)}=A_{j+1}^{r(r+p)}$ with $1\leq p\leq N-r$, $1\leq j\leq \alpha_{r+p}-1$ (the $j$-th upper off-diagonals of the $p$-th upper off-diagonal blocks of $\mathcal{X}$),
\item[\ldots]
\item $(\mathcal{X}_{11})_{1(\alpha_1-1)}=A_{\alpha_1-1}^{11}$ (the last entry in the first row of $\mathcal{X}_{11}$).
\end{itemize}

\begin{proof}[Proof of Lemma \ref{EqT}]
Since $B_j^{r}$ is symmetric for $\alpha_r-j$ odd and skew-symmetric otherwise, one easily verifies that 
\small
\begin{align}\label{ETTE}
E_{\alpha_r}(I_{m_r})\big(T_a(B_0^{r},B_1^{r},\ldots,B_{\alpha_r-1}^{r})\big)^{T}E_{\alpha_r}(I_{m_r})
 & =
\left\{
\begin{array}{ll}
\hspace{-2mm}T_a\big(-(B_0^{r})^{T},(B_1^{r})^{T},\ldots,(-1)^{\alpha_r}(B_{\alpha_r-1}^{r})^{T}\big), & 
%N+1\leq r\leq N+M
\hspace{-1mm}\alpha_r \textrm{ even}\\
\hspace{-2mm}T_a\big((B_0^{r})^{T},-(B_1^{r})^{T},\ldots,(-1)^{\alpha_r-1} (B_{\alpha_r-1}^{r})^{T}\big), & 
\hspace{-1mm}\alpha_r \textrm{ odd}
%1\leq r\leq N
\end{array}
\right.\nonumber \\
 & =
T_a\big(B_0^{r},B_1^{r},\ldots,B_{\alpha_r-1}^{r}\big).
\end{align}
\normalsize
Therefore
\[
(\mathcal{F}\mathcal{X}^{T}\mathcal{F}\mathcal{B} \mathcal{X})^{T}=\mathcal{X}^{T}\mathcal{B}^{T}\mathcal{F}\mathcal{X}\mathcal{F} =
%\mathcal{F}\mathcal{F}\mathcal{X}^{T}\mathcal{F}(\mathcal{F}\mathcal{B}^{T}\mathcal{F})\mathcal{X}\mathcal{F}=
\mathcal{F}(\mathcal{F}\mathcal{X}^{T}\mathcal{F}\mathcal{B}\mathcal{X})\mathcal{F},
\]
so it suffices to compare the blocks in the upper triangular parts of $\mathcal{F}X^{T}\mathcal{F}\mathcal{B} X$ and $\mathcal{C}$.  
Since these blocks are rectangular upper triangular Toeplitz of the same size,
it is further enough to compare their first rows.
By simplifying the notation with $\mathcal{Y}:=\mathcal{B}\mathcal{X}$ and $\widetilde{\mathcal{X}}:=\mathcal{F}X^{T}\mathcal{F}$, 
the equation (\ref{eqFYFIY}) reduces to a system of equations:
\begin{align}
\label{f1}
%((\widetilde{\mathcal{X}}\mathcal{Y})_{r(r+p)})_{1j}
(\mathcal{C}_{r(r+p)})_{1j}= & \big((\widetilde{\mathcal{X}}\mathcal{Y})_{r(r+p)}\big)_{1j} \qquad \qquad \qquad (1\leq j\leq \alpha_{r+p},\quad 0\leq p \leq N-r)  \nonumber\\
= & (\widetilde{\mathcal{X}}_{rr})_{(1)}(\mathcal{Y}_{r(r+p)})^{(j)}+\sum_{k=r+1}^{N}(\widetilde{\mathcal{X}}_{rk})_{(1)}(\mathcal{Y}_{k(r+p)})^{(j)}
%\vspace{-1mm}
 +\sum_{k=1}^{r-1}(\widetilde{\mathcal{X}}_{rk})_{(1)}(\mathcal{Y}_{k(r+p)})^{(j)}.
  %\qquad 1<r<N, .
%  \nonumber
\end{align}

First, to calculate $A_0^{rr}$ for $r\in \{1,\ldots,N\}$ we observe (\ref{f1}) for $p=0$, $j=1$. Since
%\vspace{-1mm}
\[
(\widetilde{\mathcal{X}}_{rk})_{(1)}=\left\{
\begin{array}{ll}
\begin{bsmallmatrix}
(A_0^{kr})^{T} & * & \ldots & *
\end{bsmallmatrix}, & k\geq r\\
\begin{bsmallmatrix}
0 & * & \ldots & *
\end{bsmallmatrix}, & k<r
\end{array}
\right.,\qquad
(\mathcal{Y}_{kr})^{(1)}=
\left\{
\begin{array}{ll}
\begin{bsmallmatrix}
B_0^{k}A_0^{kr} \\
0\\
\vdots\\
0
\end{bsmallmatrix}, & k\leq r\\
0, & k>r
\end{array}
\right.,
\]
\vspace{-1mm}
we deduce $\sum_{k=1}^{N}(\widetilde{\mathcal{X}}_{rk})_{(1)}((\mathcal{Y})_{kr})^{(1)}=
(A_0^{rr})^{T}B_0^{r}A_0^{rr}$, and it implies:
\begin{equation}\label{GABA}
C_0^{r}=(A_0^{rr})^{T}B_0^{r}A_0^{rr}, \qquad r\in \{1,\ldots,N\}.
\end{equation}
To get (\ref{EqT2}), we arbitrarily set $\mathcal{X}_{rs}$ for $N\geq 2$, $r>s$ (below the main diagonal of $\mathcal{X}$).

Next, we inductively compute the remaining entries.
We fix $p\in \{0,\ldots, N-1 \}$, $j\leq \alpha_r-1$ (but not $p=j=0$) and solve (\ref{f1}) to get $A_j^{r(r+p)}$ (loops $j=0:\alpha_1-1$, $p=1:N-1$ 
%of the algorithm 
in (\ref{EqT3ca})),
while assuming that we have already determined $A_{n}^{rs}$ for 
%
%\small
%\vspace{-1mm}
\begin{align}\label{induA1}
&j\geq 1, n\in \{0,\ldots,j-1\}, s\geq r 
\quad\textrm{or}\quad p\geq 1, n=j, r\leq s\leq r+p-1 \\
&\textrm{ or } \quad s\leq r, n\in \{0,\ldots,b_{rs}-1\}, N\geq 2, \qquad (1 \leq r, s\leq N).\nonumber
\end{align}
\normalsize

We have
\begin{align*}
\widetilde{\mathcal{X}}_{rk}
=
E_{\alpha_r}(I_{m_r})\mathcal{X}_{kr}^{T} E_{\alpha_k}(I_{m_k})
&=
%(E_{m_r}(I_{\alpha_r})Y_{rs} E_{m_s}(I_{\alpha_s}))^{T}=
\left\{
\begin{array}{ll}
\hspace{-1mm}\begin{bsmallmatrix}
\widetilde{\mathcal{T}}_{rk}\\
0
\hspace{-1mm}\end{bsmallmatrix}, 
& \alpha_r >\alpha_k \\
\hspace{-1mm}\begin{bsmallmatrix}
0 & \widetilde{\mathcal{T}}_{rk}
\end{bsmallmatrix}, & \alpha_r<\alpha_k\\
\widetilde{\mathcal{T}}_{rk}, & \alpha_r=\alpha_k
\end{array}
\right.\hspace{-1mm}, \quad
\widetilde{T}_{rk}=
T\big((A_0^{kr})^{T},\ldots,(A_{b_{kr}-1}^{kr})^{T}\big),
%,\nonumber\\
\end{align*}
%
%
%\vspace{-1mm}
%For simplicity we set
%$\Phi_{n}^{ks}:=\sum_{i=0}^{n} B_{n-i}^{k}A_i^{ks}$, $n\in \{0,\ldots, b_{rs}-1\}$, and we have
%
\begin{align}\label{YSP}
&\mathcal{Y}_{ks} 
=
\left\{
\begin{array}{ll}
\hspace{-1mm}\begin{bsmallmatrix}
\mathcal{S}_{ks}\\
0
\end{bsmallmatrix}, 
& \alpha_k >\alpha_s \\
\hspace{-1mm}\begin{bsmallmatrix}
0 & \mathcal{S}_{ks}
\end{bsmallmatrix}, & \alpha_k<\alpha_s\\
\hspace{-1mm}\mathcal{S}_{ks}, & \alpha_k=\alpha_s
\end{array}
\right. \hspace{-2mm}, \quad  
%\begin{array}{rl}
\mathcal{S}_{ks} 
% =T_a\big(B_0^{k},\ldots, B_{b_{ks}-1}^{k}\big)T\big(A_0^{ks},\ldots,A_{b_{ks}-1}^{ks}\big)\\
=T_a\big(\Phi_0^{ks},\ldots, \Phi_{b_{ks}-1}^{ks}\big),\quad
%                 \end{array}.
%\phi_0^{r},\ldots, C_{b_{rs}-1}^{r}\in \mathbb{C}^{m_r\times m_r},
%C_{k}^{rs}=\sum_{j=0}^{k} B_{k-j}^{r}A_j^{rs}.%\\
\Phi_{n}^{ks}:=\sum_{i=0}^{n} B_{n-i}^{k}A_i^{ks}.
\end{align}
To simplify the computations we set ($k,r,s\in \{1,\ldots,N\}$, $n\in \{0,\ldots, b_{rs}-1\}$):
\begin{align}
%\small
\label{Prsk}
\Psi^{krs}_{n}:=
&
\left\{
\begin{array}{ll}
\hspace{-1mm} 
\begin{bsmallmatrix}
(A_0^{kr})^{T} & (A_1^{kr})^{T} & \ldots & (A_{n}^{rr})^{T} 
\end{bsmallmatrix}
\begin{bsmallmatrix}
\Phi_{n}^{ks} \\
-\Phi_{n-1}^{ks} \\
\vdots \\
%(-1)^{n-1}\Phi_{1}^{ks} \\
(-1)^{n}\Phi_{0}^{ks} \\
\end{bsmallmatrix}, & \hspace{-1mm} n\geq 0
\\
\hspace{-1mm} 0, & \hspace{-1mm} n<0
\end{array}
\right.\\
=
&
\left\{
\begin{array}{ll}
\hspace{-1mm}  \sum_{i=0}^{n}(-1)^{i}(A_i^{kr})^{T}\Phi_{n-i}^{ks}, & \hspace{-1mm}n\geq 0\\
\hspace{-1mm} 0, & \hspace{-1mm} n<0
\end{array}
\right.\hspace{-2mm} \nonumber
\end{align}
\normalsize
and by using the fact $(B_n^{r})^{T}=(-1)^{\alpha_r-n+1}$, we observe the following:
%
%\small
\begin{align}\label{simetrija}
(\Psi^{krs}_{n})^{T}
&
%=\sum_{i=0}^{n}(-1)^{n-i}(\Phi_i^{ks})^{T}A_{n-i}^{kr} 
=\sum_{i=0}^{n}(-1)^{n-i}(A_l^{ks})^{T}\big(\sum_{l=0}^{i}( B_{i-l}^{k})^{T}A_{n-i}^{kr}\big)\nonumber
%\\
%&
=\sum_{l=0}^{n}\sum_{i=l}^{n}(-1)^{n+\alpha_k+l+1}(A_l^{ks})^{T} B_{i-l}^{k}  A_{n-i}^{kr}
\\
&
=(-1)^{\alpha_k+1}\sum_{l=0}^{n}(-1)^{n-l}(A_{l}^{ks})^{T} \sum_{i'=0}^{n-l} B_{i'}^{k}A_{n-l-i'}^{kr}
%\\
%&=(-1)^{n}\sum_{l=0}^{n}(-1)^{l}(A_l^{ks})^{T}\Phi_{n-l}^{kr}
=(-1)^{\alpha_k+n+1}\Psi^{ksr}_{n}.
% \qquad n\geq 0.  
%\nonumber
\end{align}
%\normalsize
% 
Furthermore,
%\vspace{-2mm}
\begin{align}\label{psijr}
(\widetilde{\mathcal{X}}_{rk})_{(1)}(\mathcal{Y}_{k(r+p)})^{(n+1)}
&=
\left\{
\begin{array}{ll}
%\Psi_j^{rrr},                            &  p=0, k=r,\\
\Psi_{n-\alpha_{r+p}+\alpha_k}^{kr(r+p)}, &  k\geq r+p+1\\
\Psi_{n}^{kr(r+p)},                       &  r+p \geq k\geq r+1, p\geq 1\\
\Psi_{n-\alpha_k+\alpha_r}^{kr(r+p)},     &  k\leq r\\
\end{array}
\right..
\end{align}
In particular, for $\xi_j^{r(r+p)}$ defined as in the algorithm in (\ref{EqT3ca}), we deduce:
%
%\small
\begin{align}
\label{xijrp}
(\widetilde{\mathcal{X}}_{rr})_{(1)}(\mathcal{Y}_{r(r+p)})^{(j+1)}
=\xi_j^{r(r+p)}+
\left\{\begin{array}{ll}
\hspace{-2mm}(A_0^{rr})^{T}B_0^{r}A_j^{rr}+ (-1)^{j}(A_j^{rr})^{T}B_0^{r}A_0^{rr}, &  \hspace{-1mm}
%j\neq 0 \textrm{ even}, 
p=0\\
%\hspace{-2mm}(A_0^{rr})^{T}B_0^{r}A_j^{rr}-(A_j^{rr})^{T}B_0^{r}A_0^{rr}, & \hspace{-1mm} j \textrm{ odd}, p=0\\
\hspace{-2mm}(A_0^{rr})^{T}B_0^{r}A_j^{r(r+p)}, & \hspace{-1mm} p\geq 1
\end{array}
\right.\hspace{-1mm}. 
\end{align}
\normalsize

Next, we have
\small
\begin{align*}
\begin{array}{l}
(\widetilde{\mathcal{X}}_{rk})_{(1)}=
\left\{
\begin{array}{ll}
\hspace{-1mm}\begin{bsmallmatrix}
(A_0^{kr})^{T} & (A_1^{kr})^{T} & \ldots & (A_{\alpha_{r}-1}^{kr})^{T}
\end{bsmallmatrix}, & \hspace{-1mm} k\geq r\\
\\
\hspace{-1mm}\begin{bsmallmatrix}
0& \ldots & 0 &(A_0^{kr})^{T}  & \ldots &(A_{b_{kr}}^{kr})^{T} 
\end{bsmallmatrix}, & \hspace{-1mm} k<r
\end{array}
\right.\hspace{-1mm},
\\
\\
(\mathcal{Y}_{(r+p)(r+p)})^{(j+1)}=
\begin{bsmallmatrix}
\Phi_j^{r(r+p)} \\
-\Phi_{j-1}^{r(r+p)} \\
\vdots\\
(-1)^{j}\Phi_0^{r(r+p)}
\end{bsmallmatrix},
\end{array}
% \quad
%
\hspace{-2mm}
(\mathcal{Y}_{k(r+p)})^{(j+1)}=
\left\{\begin{array}{ll}
\hspace{-1mm}\begin{bsmallmatrix}
\Phi_j^{r(r+p)} \\
%-\Phi_{j-1}^{r(r+p)} \\
\vdots\\
(-1)^{j}\Phi_0^{r(r+p)}\\
0\\
\vdots\\
0
\end{bsmallmatrix}, & \hspace{-4mm} r+p > k\\
%&(\mathcal{Y}_{(r+p)(r+p)})^{(j+1)}=
\hspace{-1mm} \begin{bsmallmatrix}
\Phi_{j-\alpha_{r+p}-\alpha_k}^{r(r+p)} \\
%-\Phi_{j-\alpha_{r+p}-\alpha_k-1}^{r(r+p)} \\
\vdots\\
(-1)^{j-\alpha_{r+p}-\alpha_k}\Phi_0^{r(r+p)}\\
0\\
\vdots\\
0
\end{bsmallmatrix}, & \hspace{-2mm} k> r+p
%\\
%&(\mathcal{Y}_{k(r+p)})^{(j+1)}=
\end{array}
\right.
\hspace{-2mm}.
\end{align*}
\normalsize
Hence, the last two terms in (\ref{f1}) for $j+1$ instead of $j$ ($N\geq r+1\geq 2$) are:
%\vspace{-1mm}
%
\begin{align}\label{thjrp}
\Xi (j,r,p):=
& \sum_{k=r+1}^{N}(\widetilde{\mathcal{X}}_{rk})_{(1)}(\mathcal{Y}_{k(r+p)})^{(j+1)}\\
=
&\left\{\begin{array}{ll}
\sum_{k=r+1}^{N}\Psi_{j-\alpha_{r}+\alpha_k}^{krr}, & j\geq 1, p=0\\
\sum_{k=r+1}^{r+p}\Psi_{j}^{kr(r+p)}
+\sum_{k=r+p+1}^{N}\Psi_{j-\alpha_{r+p}+\alpha_k}^{kr(r+p)}, & j\geq 0, p\geq 1. %r+p+1\leq N\\
%\sum_{k=r+1}^{r+p}\Psi_{j}^{r(r+p)}, &p\geq 1, r+p= N  \\
%0,   &     \textrm{otherwise}
\end{array}
\right.,\nonumber\\
%\end{align}
%
%\begin{align}
\label{lajrp}
\Lambda(j,r,p):=
&  \sum_{k=1}^{r-1}(\widetilde{\mathcal{X}}_{rk})_{(1)}(\mathcal{Y}_{k(r+p)})^{(j+1)}=\sum_{k=1}^{r-1}\Psi_{j-\alpha_k+\alpha_r}^{kr(r+p)}.
\end{align}
For $j,p\geq 0$ with $j+p\geq 1$ we define
\begin{equation}\label{Djrp}
D_j^{r(r+p)}:=\xi_j^{r(r+p)}+\Xi(j,r,p)+\Lambda(j,r,p).
\end{equation}
We combine (\ref{f1}) with (\ref{xijrp}), (\ref{thjrp}), (\ref{lajrp}), (\ref{Djrp}) to obtain:
\begin{align}
\label{eqATB}
&(A_0^{rr})^{T}B_0^{r}A_j^{r(r+p)}=-D_j^{r(r+p)},\qquad p\geq 1,\\
&(A_0^{rr})^{T}B_0^{r}A_j^{rr}+(-1)^{j}(A_j^{rr})^{T}B_0^{r}A_0^{rr}=C_j^{r}-D_j^{rr}, \qquad j\geq 1\,\,\,\,\, (p=0).\nonumber
\end{align}
From (\ref{Prsk}) we get 
%$\Psi^{rrr}_{j}=(-1)^{\alpha_r-j+1}(\Psi^{rrr}_{j})^{T}$, 
$\Psi^{krr}_{j-\alpha_r+\alpha_k}=(-1)^{\alpha_r-j+1}(\Psi^{krr}_{j-\alpha_r+\alpha_k})^{T}$, thus $\xi_j^{rr}$, $\Xi(j,r,0)$, $\Lambda(j,r,0)$, $C_j^{r}-D_j^{rr}$ are skew-symmetric (symmetric) for $\alpha_r-j$ odd (even).

Since (\ref{GABA}) implies 
$A_0^{rr}(C_0^{r})^{-1}=
((A_0^{rr})^{T}B_0^{r})^{-1}$, 
the first equation of
(\ref{eqATB}) yields $A_j^{r(r+p)}=-A_0^{r}(C_0^{r})^{-1}D_j^{r(r+p)}$ for $p\geq 1$.
Next, we get $A_j^{rr}$ by solving the second equation (\ref{eqATB}). It is of the form $A^{T}X+(-1)^{j}X^{T}A=B$
for either $j$ odd and $B$ skew-symmetric or $j$ even and $B$ symmetric. 
The solution of this equation is $X=(A^{T})^{-1}(Z+\frac{1}{2}B)$ with $Z$ symmetric for $j$ odd and $Z$ skew-symmetric for $j$ even.
We have $A=B_0^{r} A_0^{rr}$ (hence
$(A^{T})^{-1}=
%((A_0^{rr})^{T}B_0^{r})^{-1}=
A_0^{rr}(C_0^{r})^{-1}$)
and $B=C_j^{r}-D_j^{rr}$, which depend only on $A_n^{rs}$ with $n,r,s$ satisfying (\ref{induA1}).

It is only left to sum up the dimension (\ref{cdim}):
\small
\begin{align*}
%\dim_{\mathbb{C}} (\Sigma_{\mathcal{K}_{}^{}})=\left\{
%\begin{array}{ll}
%\displaystyle
%\vspace{-2mm}
%
&\sum_{r=1}^{N}\sum_{s=1}^{r-1}\alpha_s m_r m_s
+ \hspace{-1.5mm}\sum_{\alpha_r \textrm{ even}}\hspace{-1.5mm}\big(
%\frac{1}{2} \widetilde{m}_r(\widetilde{m}_r-1)+
\tfrac{\alpha_r}{4}(m_r^2-m_r)+\tfrac{\alpha_r}{4}(m_r^2+m_r)\big)+\hspace{-1.5mm}
\sum_{\alpha_r \textrm{ odd}}\hspace{-1.5mm}\big(
%\tfrac{1}{2}\widetilde{m}_r(\widetilde{m}_r-1)+
\tfrac{\alpha_r-1}{4}(m_r^2-m_r)+
\tfrac{\alpha_r+1}{4}(m_r^2+m_r)\big).
%\\
\end{align*}
\normalsize
%&
%\qquad \qquad \quad +\sum_{\alpha_r \textrm{ odd}}\hspace{-1mm}\big(
%\tfrac{\alpha_r-1}{4}(m_r^2-m_r)+
%\tfrac{\alpha_r+1}{4}(m_r^2+m_r)\big)
%\\
%
%\begin{align*}
%&\qquad =
A straightforward simplification yields 
$\sum_{r=1}^{N}\big(\tfrac{1}{2}m_r^{2}\alpha_r+\sum_{s=1}^{r-1}\alpha_s m_r m_s\big) -
\sum_{\alpha_r \textrm{ odd}}\tfrac{1}{2} m_r$.
%\end{align*}
%
This completes the proof of the lemma.
\end{proof}

\begin{example}
We solve $\mathcal{F}\mathcal{X}^{T}\mathcal{F}\mathcal{B}\mathcal{X}=\mathcal{B}$ for $\mathcal{F}=E_3(I)\oplus E_2(I)$, $\mathcal{B}=\mathcal{C}=(J\oplus -J\oplus J)\oplus (I \oplus -I)$ with the identity $I$ and nonsingular skew-symmetric $J$, and $\mathcal{X}$ partitioned conformally to $\mathcal{B}$. 
We have:
\small
\begin{align*}
&\qquad
\mathcal{B}=(\mathcal{F}\mathcal{X}^{T}\mathcal{F})(\mathcal{B}\mathcal{X})
=\begin{bmatrix}[ccc|c c]
A_1^{T} & B_1^{T} & C_1^{T}  &  P^{T}  &  Q^{T} \\  % &  R_1^{T}\\
0   & A_1^{T} & B_1^{T}     &  0      &   P^{T} \\ %  &  0\\
0   & 0   & A_1^{T}      &  0   & 0   \\  %  & 0 \\
%0   & 0   &  0  &  A_1^{T}  & 0 &0 & 0\\
\hline
0    & M^{T}  & N^{T}    &  A_2^{T}  &  B_2^{T} \\  % &  R_2^{T}\\
 0   & 0     & M^{T}     &  0       &   A_2 ^{T}  %  &  0\\
%\hline
%0   & 0     & J_1^{T}                         &  0    &   J_2^{T}    & A_2 
\end{bmatrix}
\begin{bmatrix}[ccc|cc]
JA_1 & JB_1 & JC_1  &  JM  &  JN \\  % &  J_1\\
0   & -JA_1 & -JB_1    &  0  &   -J M  \\  % &  0\\
0   & 0   & J A_1     &  0  & 0   \\  %  & 0 \\
%0   & 0   &  0  &  \overline{A}_1   & 0 &0 & 0\\
\hline
0     & P  & Q                &  A_2    &  B_2 \\  % &  J_2\\
0     & 0   & -P              &  0    &  -A_2    %  &  0\\
%\hline
%0     & 0   & R_1                        &  0    &   R_2  %  & A_3 
\end{bmatrix}=\\
%\end{align*}
%\scriptsize
%\small
%\setlength{\arraycolsep}{3pt}
%\begin{align*}
&=
\begin{bmatrix}[ccc|cc]
A_1^{T}JA_1 & A_1^{T}JB_1\hspace{-0.5mm}-\hspace{-0.5mm}B_1^{T} J A_1 & A_1^{T}J C_1\hspace{-0.5mm}+\hspace{-0.5mm}C_1^{T}J A_1\hspace{-0.5mm}-\hspace{-0.5mm}B_1^{T}J  B_1    & \hspace{-0.5mm} A_1^{T}JM\hspace{-0.5mm}+\hspace{-0.5mm}P^{T}A_2 \hspace{-0.5mm} &  A_1^{T}J N+B_1^{T}JM  \\ 
& & + P^{T}Q - Q^{T} P & & + P^{T}B_2- Q^{T}A_2 \\ % &  N_1^{*}J_2+R_1^{T}A_3\\
%           &       & +B^{T} B  - G^{T}G               &                        \\  %   &      +N_1^{*}B_2     +P_1^{T}\overline{A}_2   \\  % & +A_1^{*}J_1 \\
  0         & -A_1^{T}JA_1 &  -A_1^{T} J B_1 +B_1^{T}J A_1    &  0   &   -A_1^{T}JM - P^{T}A_2 \\ % &  0\\
%   &                       &                          &       + B_1^{*}B_1+N_1^{T}\overline{N}_1    &     &                          &   \\
0   &     0                  & A_1^{T}J A_1                        &  0  & 0             \\  %         &      0 \\
%   &    &    &  A_1^{T}\overline{A}_1  &  0   &0 & 0\\
\hline 
   &    &                     &  A_{2}^{T} A_2       &   A_2^{T}B_2 - B_2^{T} A_2 \\ 
	 &    &                     &                       &     -M^{T}JM                          \\
	% &  A_2^{*}J_2+R_2^{T}A_3\\
   &    &                       &      0           &    + A_2^{T}A_2         %   &                       \\
%   &    &             &   0      &   A_2^{T}\overline{A}_2       % &  0\\
%\hline
%   &    &                     &      &      & A_3 ^{T}A_3
\end{bmatrix}
\end{align*}
\normalsize
By comparing diagonals of the main diagonal blocks we get $J=(A_1)^{T} J A_1$, $I=(A_2)^{T}A_2$ ($A_1$ $J$-pseudo-orthogonal, $A_2$ orthogonal). Next, choose $P,Q$ arbitra\-rily. The diagonal element of the upper right block gives $A_1^{T}JM+P^{T}A_2=0$, 
thus $M=A_1^{} J^{-1}P^{T}A_2$.

The first super diagonals yield $A_1^{T}JB_1-B_1^{T} J A_1=0$, $A_2^{T}B_2 - B_2^{T} A_2-M^{T}JM=0$ and $A_1^{T}JN-B_1^{T}M+P^{T}B_2-Q^{T}A_2$, thus $B_1=A_1J^{-1}Z_1$ with $Z_1$ skew-symmetric, $B_2=\frac{1}{2}A_2J^{-1}M^{T}JM+ A_2J^{-1}Z_2$ with $Z_2$ symmetric, $N=A_1 J^{-1}(B_1^{T}M-P^{T}B_2+Q^{T}A_2)$, respectively.
Finally, the second supper diagonal yields $A_1^{T}J C_1+C_1^{T}J A_1-B_1^{T}J B_1 + P^{T}Q - Q^{T} P=0$, therefore $C_1=\frac{1}{2}A_1 J^{-1}(B_1^{T}J B_1 - P^{T}Q + Q^{T} P)+A_1 J^{-1} Z_2
%=-\frac{1}{2}A_1(Z_1^{T}Z_1)+A_1 Z_2
$ with symmetric $Z_3$.
\end{example}

Solutions of the equation (\ref{eqFYFIY}) with $\mathcal{C}=\mathcal{B}$ have a nice decomposition.

\begin{lemma}\label{lemauni}
Let $\mathbb{T}_{}^{\alpha,\mu}$ for $\alpha=(\alpha_1,\ldots,\alpha_n)$ with $\alpha_{1}>\ldots >\alpha_{N}$ and $\mu=(m_1,\ldots,m_N)$ be as 
in (\ref{0T0}), and let $\mathbb{X}_{}\subset \mathbb{T}_{}^{\alpha,\mu}$ be the set of solutions of the equation (\ref{eqFYFIY}) for $\mathcal{C}=\mathcal{B}$ as in (\ref{BBF}).
Then $\mathbb{X}$ is a group, more precisely, a semidirect product of subgroups 
\[
\mathbb{X}_{}=\mathbb{O}_{}\ltimes \mathbb{V}_{}\subset \mathbb{T}_{}^{\alpha,\mu},
\]
in which $\mathbb{O}_{}$ consists of all matrices $\mathcal{Q}=\bigoplus_{r=1}^{N}\big(\bigoplus_{j=1}^{\alpha_r} Q_{r}\big)$ for $Q_{r}\in \mathbb{C}^{m_r\times m_r}$ such that $
Q_{r}^{T}B_0^{r}Q_{r}=B_0^{r}:=[\mathcal{B}_{rr}]_{11}$, while any $\mathcal{V}\in \mathbb{V}_{}$ 
can be written as 
$\mathcal{V}=\prod_{j=0}^{n_{\mathcal{V}}}\mathcal{U}_j$, where  
$\mathcal{U}_0=\bigoplus_{r=1}^{N}\mathcal{W}_r$ with $\mathcal{W}_r$ upper unitriangular Toeplitz 
and $\mathcal{U}_1,\ldots,\mathcal{U}_{n_{\mathcal{V}}}$ are of the form (\ref{0T0}) 
with (\ref{Hptk}). Moreover, if $\mathcal{B}=\bigoplus_{r=1}^{N}\big(\bigoplus_{j=1}^{\alpha_r} (-1)^{j} B_0^{r}\big)$, then $\mathcal{U}_0$ is of the form (\ref{asZ}), and $\mathcal{U}_1,\ldots,\mathcal{U}_{n_{\mathcal{V}}}$
are of the form (\ref{0T0}) with (\ref{Hptk}), (\ref{Hptk2}) for $B_r=B_0^{r}$. (The subgroup $\mathbb{V}_{}$ is normal and unipotent of order at most $\leq \alpha_1-1$.) 
\end{lemma}

%Solutions of (\ref{eqFYFIY}) with $\mathcal{C}=\mathcal{B}$ form a group; 
%if $\mathcal{X}_1,\mathcal{X}_2$ the product $\mathcal{X}_1\mathcal{X}_2^{-1}$ is a solution.
%
%\begin{align*}
%\mathcal{F}(\mathcal{X}_1\mathcal{X}_2^{-1})^{T}\mathcal{F}\mathcal{B} (\mathcal{X}_1\mathcal{X}_2^{-1})
%&=\mathcal{F}(\mathcal{X}_2^{-1})^{T}\mathcal{F}(\mathcal{F}\mathcal{X}_1^{T}\mathcal{F}\mathcal{B} \mathcal{X}_1)\mathcal{X}_2^{-1}
%=\mathcal{F}(\mathcal{X}_2^{-1})^{T}\mathcal{F}\mathcal{B}\mathcal{X}_2^{-1}=\\
%&=\mathcal{F}(\mathcal{X}_2^{-1})^{T}\mathcal{F}\mathcal{B}(\mathcal{B}^{-1}\mathcal{F}\mathcal{X}_2^{T}\mathcal{F}\mathcal{B})=\mathcal{B}.
%\end{align*}

The proof of the lemma works mutatis mutandis as the proof of \cite[Lemma 4.2]{TSOS} 
%of the form (\ref{0T0}) 
%for $\mathcal{C}=\mathcal{B}$ 
with upper triangular Toeplitz diagonal blocks (not alternating).

\section{Proofs of the main results}\label{sec2}

Before we begin with the proofs of our theorems we recall the classical result \cite[Ch. VIII]{Gant} on 
solutions of a homogeneous Sylvester equation.

\begin{proposition} \label{resAoXXA} 
Suppose $\mathcal{J}$ is the Jordan canonical form (with direct summands of the form (\ref{Jblock})) and let us consider the following matrix equation:
\begin{align}\label{eqJYYJ}
\mathcal{J}X=X\mathcal{J}
\end{align}  
\begin{enumerate}[label={\arabic*.},ref={\arabic*},
%itemsep=1pt,itemindent=0pt,
leftmargin=20pt]
\item \label{resAoXXA1}
Let further $\lambda_1,\ldots,\lambda_n\in \mathbb{C}$ be pairwise distinct eigenvalues of $\mathcal{J}=\bigoplus_{j=1}^{n}\mathcal{J}_j$,
in which $\mathcal{J}_j$ consists of all summands that correspond to the eigenvalue $\lambda_j$.
Then the solution of (\ref{eqJYYJ}) is of the form $Y=\bigoplus_{j=1}^{n} X_j$ with $X_j$ a solution of $\mathcal{J}_j X_j=X_j\mathcal{J}_j$.

\item \label{resAoXXA2}
Assume $\mathcal{J}=\bigoplus_{r=1}^{N}\big(\bigoplus_{j=1}^{m_r} J_{\alpha_r}(\lambda)\big)$, $\lambda \in \mathbb{C}$ (i.e. $\mathcal{J}$ has precisely one eigenvalue $\lambda$) and $\alpha_1>\ldots>\alpha_N$. Then the solution of (\ref{eqJYYJ}) is $X=[X_{rs}]_{r,s=1}^{N}$, where every block $X_{rs}$ is further an $m_r$-by-$m$ block matrix with $\alpha_r$-by-$ \alpha_s$ blocks of the form
\begin{equation}\label{QTY}
\left\{\begin{array}{ll}
\begin{bmatrix} 
0 & T
\end{bmatrix}, & \alpha_r<\alpha_s \\
\begin{bmatrix}
T\\
0
\end{bmatrix}, & \alpha_r>\alpha_s\\
T, & \alpha_r=\alpha_s
\end{array}
\right.,
\end{equation}
with $T$ upper triangular Toeplitz matrix of size $b_{rs}\times b_{rs}$; $b_{rs}:=\min\{\alpha_r,\alpha_s\}$. 
\end{enumerate}
\end{proposition}

Observe that Proposition \ref{resAoXXA} (\ref{resAoXXA2}) remains valid if we replace all $J_{\alpha_r}(\lambda)$ by $e^{J_{\alpha_r}(\lambda)}$.

\begin{lemma}\label{resAoXXAo} 
Assume $\mathcal{J}=\bigoplus_{r=1}^{N}\big(\bigoplus_{j=1}^{m_r} e^{J_{\alpha_r}(\lambda)}\big)$, $\lambda \in \mathbb{C}$.
Then the solution of the equation $\mathcal{J}X=X\mathcal{J}$ coincides with the solution $X$ in the conclusion of Proposition \ref{resAoXXA} (\ref{resAoXXA2}). 
\end{lemma}

\begin{proof}
%For any $\alpha$ 
We have $e^{J_{\alpha}(\lambda)}=e^{\lambda }e^{J_{\alpha}(0)}$ with $e^{J_{\alpha}(0)}=\sum_{j=0}^{\alpha }\frac{1}{\alpha !}\big(J_{\alpha}(0)\big)^{j}$ and $e^{J_{\alpha}(0)}=S_{\alpha}J_{\alpha}(1)S_{\alpha}^{-1}$ for some transition matrix $S_{\alpha}$. Hence $e^{J_{\alpha_r}(\lambda)}Y =Y e^{J_{\alpha_s}(\lambda)}$ is equivalent to
\begin{equation}\label{eJ0XXeJ0}
e^{J_{\alpha_r}(0)}Y =Y e^{J_{\alpha_s}(0)},
\end{equation}
and it further transforms to
\begin{equation}\label{J1XXJ1}
J_{\alpha_r}(1)\widetilde{Y} =\widetilde{Y} J_{\alpha_s}(1), \qquad \widetilde{Y}=S^{-1}_{\alpha_r}YS_{\alpha_s}.
\end{equation}
By Proposition \ref{resAoXXA} (\ref{resAoXXA2}) every solution $\widetilde{Y}$ of (\ref{J1XXJ1}) is of the form (\ref{QTY}), 
so the dimension of the solution space of (\ref{J1XXJ1}) (as well as (\ref{eJ0XXeJ0})) is $b_{rs}=\min\{\alpha_r,\alpha_s\}$. On the other hand, since upper triangular Toeplitz matrices commute, it is easy to verify that any matrix of the form (\ref{QTY}) solves (\ref{eJ0XXeJ0}).
Since these solutions form a subspace of (maximal) dimension $b_{rs}$, it is thus equal to the whole solution space. 
\end{proof}

We also recall a technical lemma \cite[Lemma 3.3 (1)]{TSOC} (cf. \cite[Sec. 3.1]{Lin} and \cite[Sec. 2]{TSOS}), which enables us to transform solutions of a Sylvester equation (block matrices with upper triangular Toeplitz blocks) into block upper triangular Toeplitz matrices (of the form (\ref{0T0})). It relies on the permutation matrices
\begin{equation}\label{perS}
\Omega_{\alpha,m}:=\left[e_1\;e_{\alpha+1}\;\ldots\;e_{(m-1)\alpha+1}\;e_2\;e_{\alpha+2}\;\ldots\;e_{(m-1)\alpha+2}\;\ldots\;e_{\alpha}\;e_{2\alpha}\;\ldots\;e_{\alpha m}\right],
%\in \mathbb{C}^{\alpha m\times \alpha m}.
%\Omega_r:=\left[e_1\;e_{\alpha_r+1}\;\ldots\;e_{(m_r-1)\alpha_r+1}\;e_2\;e_{\alpha_r+2}\;\ldots\;e_{(m_r-1)\alpha_r+2}\;\ldots\;e_{\alpha_r}\;e_{2\alpha_r}\;\ldots\;e_{\alpha_rm_r}\right],
\end{equation}
where $e_1,
%e_2,
\ldots,e_{\alpha m}$ is the standard orthonormal basis in $\mathbb{C}^{\alpha m}$. Post-multiplication by $\Omega_{\alpha,m}$ (pre-multiplication by $\Omega_{\alpha,m}^{T}$) puts the $k$-th, the $(\alpha+k)$-th, \ldots, the $((m-1)\alpha+k)$-th column (row) together for all $ k \in \{1,\ldots, \alpha\}$. For example (cf. Theorem \ref{stabr}):

%\begin{example}
%$N=2$, $\alpha_1=3$, $m_1=2$, $\alpha_2=2$, $m_2=3$:
\[
\Omega_{3,2}^{T}\begin{bmatrix}[cc|cc|cc]
a_1 & b_1 & a_2 & b_2 & a_3 & b_3 \\
0   & a_1 & 0   & a_2 & 0   & a_3\\
0   & 0   & 0   & 0   & 0   &  0 \\
\hline
a_4 & b_4 & a_5 & b_5 & a_6 & b_6 \\
0   & a_4 & 0   & a_5 & 0   & a_6\\
0   & 0   & 0   & 0   & 0   &  0 
\end{bmatrix}\Omega_{2,3}
=
\begin{bmatrix}[ccc|ccc]
a_1 & a_2 & a_3 & b_1 & b_2 & b_3 \\
a_4 & a_5 & a_6   & b_4 & b_5   & b_6\\
\hline
0   & 0   & 0   & a_1   & a_2   &  a_3 \\
0 &  0 &   0 & a_4 & a_5 & a_6 \\
\hline
0   & 0 & 0   & 0 & 0   & 0\\
0   & 0   & 0   & 0   & 0   &  0 
\end{bmatrix}.
\]
%\vspace{-1mm}
%
%\end{example}
%

\begin{lemma}\label{lemaP}
Suppose
$X=[X_{rs}]_{r,s=1}^{N}$ is an $N$-by-$N$ block matrix whose block $X_{rs}=[(X_{rs})_{jk}]_{j,k=1}^{m_r,m_s}$ is an $m_r$-by-$m_s$ block matrix with blocks of size $\alpha_r\times \alpha_s$ and such that
\[
(X_{rs})_{jk}=
\left\{
\begin{array}{ll}
[0\quad T_{jk}^{rs}], & \alpha_{r}<\alpha_{s}\\
\begin{bmatrix}
T_{jk}^{rs}\\
0
\end{bmatrix}, & \alpha_{r}>\alpha_{s}\\
T_{jk}^{rs},& \alpha_{r}=\alpha_{s}
\end{array}\right., \qquad
\begin{array}{l}
%j\in \{1,\ldots m_r\},\quad k\in \{1,\ldots m_s\}\\
T_{jk}^{rs}=T(a_{0,jk}^{rs},a_{1,jk}^{rs},\ldots,a_{b_{rs}-1,jk}^{rs})\in \mathbb{C}^{b_{rs}\times b_{rs}},\\
\\
b_{rs}:=\min \{\alpha_r,\alpha_s\},
\end{array}
\]
and let $\Omega=\bigoplus_{r=1}^{N}\Omega_{\alpha_r,m_r}$. Then 
\vspace{-1mm}
\begin{align}\label{0T02}
&\mathcal{X}:=\Omega^{T}X\Omega, \qquad \mathcal{X}=[\mathcal{X}_{rs}]_{r,s=1}^{N},\quad
\mathcal{X}_{rs}=%\Omega_r^{T}Y_{rs}\Omega_s=
\left\{
\begin{array}{ll}
[0\quad \mathcal{T}_{rs}], & 
\alpha_r<\alpha_s\\
\begin{bmatrix}
\mathcal{T}_{rs}\\
0
\end{bmatrix}, & \alpha_r>\alpha_s\\
\mathcal{T}_{rs},& \alpha_r=\alpha_s
\end{array}\right.,
\end{align}
where $\mathcal{X}_{rs}$ is of size $\alpha_r \times \alpha_s$ and $\mathcal{T}_{rs}=
T(A_0^{rs},\ldots,A_{b_{rs}-1}^{rs})$ 
with $ A_{n}^{rs}:=[a_{n,jk}^{rs}]_{j,k=1}^{m_r,m_s}$.
\end{lemma}

We are ready to prove our main results. Here is the outline of the proof. 
The isotropy group at skew-symmetric or orthogonal normal form $\mathcal{A}^{}$ (under orthogonal similarity) consists 
of all orthogonal solutions $Q$ of the Sylvester equation 
%
%\vspace{-1mm}
\begin{equation}\label{HQQH}
\mathcal{A}^{}Q= Q\mathcal{A}^{}.
\end{equation}
This equation is then transformed into a simpler equation
\begin{equation}\label{JQQJ}
\mathcal{J}^{}X= X\mathcal{J}^{}, \quad  \quad
X=U Q U^{-1} \qquad \quad \quad
(\mathcal{A}=U^{-1}\mathcal{J}U^{}),
\end{equation}
where $\mathcal{J}$ is a suitable similarity normal form (e.g. the Jordan canonical form) for $\mathcal{A}$ and $U$ is a transition matrix. 
By applying Proposition \ref{resAoXXA} (\ref{resAoXXA2}) we obtain all $X$ that solve the equation (\ref{JQQJ}). 
Next, we need to figure out which of these solutions provide the orthogonal solutions $Q=U^{-1} X U$ of (\ref{HQQH}). This yields an equation:
\begin{align}\label{QK}
I=                   & \big((U^{})^{T}X^{T} (U^{-1})^{T}\big)   \big(  U^{-1} X U \big),  \qquad \qquad (I=Q^TQ),\\
(U^{-1})^{T} U^{-1}=  & X^{T} \big((U^{-1})^{T} U^{-1}\big) X.\nonumber
\end{align}
Combining the equations (\ref{JQQJ}) and (\ref{QK}) is thus at the core of the problem. 
Eventually, it will be reduced to a certain matrix equation with matrices of the form (\ref{0T0}) and we possibly solve it by using Lemma \ref{EqT}.  
This part of the proof seems to be considerably more involved than in \cite{TSOC}, \cite{TSOS}.

\begin{proof}[Proof of Theorem \ref{stabw} (\ref{stabw1})]
Suppose
\begin{equation}
\label{K1l}
\mathcal{K}^{}
%=\bigoplus_{r=1}^{N} \mathcal{H}_{\alpha_r}(\lambda)
=\bigoplus_{r=1}^{N}\big( \bigoplus_{j=1}^{m_r}  K_{\alpha_r}(\lambda)\big),\qquad \lambda\in \mathbb{C}\setminus\{0\},
%\mathcal{H}^{}
%=\bigoplus_{r=1}^{N} \mathcal{H}_{\alpha_r}(\lambda)
\end{equation}
where $K_{\alpha_r}(\lambda)$ is as in (\ref{Km}) for $m=\alpha_r$. The Jordan canonical form of $\mathcal{K}^{}$ is
\begin{align}\label{JCF1}
\mathcal{J}=\bigoplus_{r=1}^{N}\Big( \bigoplus_{j=1}^{m_r} \big( J_{\alpha_r}(\lambda)\oplus  J_{\alpha_r}(-\lambda)\big)\Big),\qquad \qquad (\mathcal{K}=S^{-1}\mathcal{J}S),%\qquad\\e
\end{align}
with the corresponding transition matrix 
\begin{align}\label{SRP}
S=\bigoplus_{r=1}^{N}\big( \bigoplus_{j=1}^{m_r} S_{2\alpha_r}\big),\quad S_{2\alpha}:=R_{2\alpha} P_{2\alpha}, \quad
%& P=\bigoplus_{r=1}^{N}\Big( \bigoplus_{j=1}^{m_r} P_{2\alpha_r}\Big),\qquad 
P_{\alpha}:=\tfrac{1}{\sqrt{2}}(I_{\alpha}+iE_{\alpha}), \,\,\,
%& 
%R=\bigoplus_{r=1}^{N}\Big( \bigoplus_{j=1}^{m_r} R_{2\alpha_r}\Big),\qquad 
R_{2\alpha}:=I_{\alpha} \oplus F_{\alpha},
% \quad F_{\alpha}=\oplus_{p=1}^{\alpha}(-1)^{p}
%\quad P_{\alpha}:=\frac{1}{\sqrt{2}}(I_{\alpha}+iE_{\alpha}),
\end{align}
where $E_\alpha :=E_\alpha (1)$ is the backward-identity matrix of size $\alpha\times \alpha$ (see (\ref{BBF})) and $F_{\alpha}:=\oplus_{k=1}^{\alpha}(-1)^{k}$ (i.e. diagonal entries alternate between minus and plus one).

We conjugate $\mathcal{J}$ by an appropriate permutation matrix to collect together blocks corresponding to the same eigenvalue ($\lambda$ or $-\lambda$).
We use an analog of (\ref{perS}) (with $\alpha=2$) 
%and $m$ replaced by $2m$) 
for block matrices. Define a $2m$-by-$2m$ block  permutation matrix whoose blocks are identities and zero matrices:
\begin{equation}\label{perS2}
\Omega_{2,m}(I_{\alpha}):=\left[e_1(I_{\alpha})\;\,e_{3}(I_{\alpha})\;\,\ldots\;\,e_{2m-1}(I_{\alpha})\;\,e_2(I_{\alpha})\;\,e_{4}(I_{\alpha})\;\,\ldots\;\,e_{2m}(I_{\alpha})\right],
%\in \mathbb{C}^{\alpha m\times \alpha m}.
%\Omega_r:=\left[e_1\;e_{\alpha_r+1}\;\ldots\;e_{(m_r-1)\alpha_r+1}\;e_2\;e_{\alpha_r+2}\;\ldots\;e_{(m_r-1)\alpha_r+2}\;\ldots\;e_{\alpha_r}\;e_{2\alpha_r}\;\ldots\;e_{\alpha_rm_r}\right],
\end{equation}
in which $e_j(I_{\alpha})$ is a column ($n$ rows) with $I_{\alpha}$ in the $j$-th row and ${\alpha}$-by-${\alpha}$ zero matrices otherwise.
Post-multiplication by $\Omega_{2,m}(I_{\alpha})$
% from the right 
(with $\Omega_{2,2m}^{T}$)
% from the left) 
puts the 1st, the 3rd, \ldots, the $(2m-1)$-th column, the 2nd, the 4th,\ldots, the $2m$-th (row) together. 
Thus:
\begin{align}\label{J1F}
\widetilde{\Omega}_1^{T}\mathcal{J}\widetilde{\Omega}_1=\bigoplus_{r=1}^{N}\Big( \bigoplus_{j=1}^{m_r} J_{\alpha_r}(\lambda)\oplus \bigoplus_{j=1}^{m_r} J_{\alpha_r}(-\lambda)\Big), \qquad \widetilde{\Omega}_1:=\bigoplus_{r=1}^{N} \Omega_{2,2m_r}(I_{\alpha_r}).
\end{align}
Similarly as in (\ref{perS2}), we can also permute blocks of possibly different dimensions:
%the following $2N$-by-$2N$ block matrix with blocks of possibly different dimensions: 
%$\Omega_{\alpha,m}'$ 
\small
\begin{align}\label{perK2}
\widetilde{\Omega}_2=\left[e_1(I_{m_1\alpha_1})\;e_{3}(I_{m_2\alpha_2})\;\ldots\;e_{2N-1}(I_{m_N\alpha_N})\;e_2(I_{m_1\alpha_1})\;e_{4}(I_{m_2\alpha_2})\;\ldots\;e_{2N}(I_{m_N\alpha_N})
\right], 
%\quad \alpha=(\alpha_1,\alpha,\alpha_N),
\end{align}
\normalsize
where $e_j(I_{m_r\alpha_r})$ is a column ($2N$ rows) with $I_{m_r\alpha_r}$ in the $j$-th row and the appropriate zero matrices otherwise.
Multiplication by $\widetilde{\Omega}_2$ (by $\widetilde{\Omega}_2^{T}$) from the right (left)
%(with its transpose from the left) 
arranges columns (rows) in the same order as would $\Omega_{2,N}(I_{\alpha})$.
From (\ref{J1F}) we get
\begin{equation}\label{J2F}
\widetilde{\mathcal{J}}:=\widetilde{\Omega}^{T}\mathcal{J}\widetilde{\Omega}=
\bigoplus_{r=1}^{N}\Big( \bigoplus_{j=1}^{m_r} J_{\alpha_r}(\lambda)\Big)\oplus 
\bigoplus_{r=1}^{N}\Big( \bigoplus_{j=1}^{m_r}  J_{\alpha_r}(-\lambda)\Big), \qquad 
\widetilde{\Omega}:=\widetilde{\Omega}_1\widetilde{\Omega}_2.
\end{equation}
%
%Using
%(\ref{SRP}), 
%(\ref{J2F})
The equation (\ref{HQQH}) for $\mathcal{A}=\mathcal{K}$ (see (\ref{K1l})) transforms to (\ref{JQQJ}) for $\mathcal{J}=\widetilde{\mathcal{J}}$, $U=\widetilde{\Omega}^{T} S$:
%\vspace{-2mm}
\begin{equation}\label{wJQQwJ}
\widetilde{\mathcal{J}}X=X\widetilde{\mathcal{J}}, \qquad X=(\widetilde{\Omega}^{T} S) Q (S^{-1}\widetilde{\Omega}), 
%\quad U:=\widetilde{\Omega}^{T} S,
%(S^{-1}\widetilde{\Omega}), 
\qquad \qquad (Q \textrm{ orthogonal}).
\end{equation}
%
%\quad
Proposition \ref{resAoXXA} gives the solution $X=X_1\oplus X_2$ of (\ref{wJQQwJ}), where $X_1,X_2$ are $N$-by-$N$ block matrices such that their blocks $(X_1)_{rs}$, $(X_2)_{rs}$ are $m_r$-by-$m_s$ block matrices whose blocks are rectangular Toeplitz of size $\alpha_r\times \alpha_s$
and of the form (\ref{QTY}).

In view of (\ref{QK}) orthogonality of $Q=(S^{-1}\widetilde{\Omega}) X (\widetilde{\Omega}^{T} S)$ yields the equation
%
%\vspace{-1mm}
\begin{align}\label{QTQH}
%I=   & \big(S^{T}\widetilde{\Omega}X^T\widetilde{\Omega}^{T}(S^{-1})^{T})\big)\big(S^{-1}\widetilde{\Omega}X \widetilde{\Omega}^{T} S \big)\\
%I=   & \big(P^TR^{T}\widetilde{\Omega}X^T\widetilde{\Omega}^{T}(R^{-1})^{T}(P^{-1})^T\big)\big(P^{-1}R^{-1}\widetilde{\Omega}X \widetilde{\Omega}^{T}R P \big)\\
\widetilde{\Omega}^{T}(S^{-1})^{T}S^{-1}\widetilde{\Omega} =   &  X^T \big(\widetilde{\Omega}^{T}(S^{-1})^{T} S^{-1} \widetilde{\Omega}\big) X.%\nonumber
%\widetilde{\Omega}^{T}(R^T)^{-1}(P^T)^{-1}P^{-1}R^{-1}\widetilde{\Omega} =   &  X^T \widetilde{\Omega}^{T}RP^{2} R \widetilde{\Omega} X\nonumber\\
%\widetilde{\Omega}^{T}(R P^{2} R)\widetilde{\Omega} = &  X^{T} \widetilde{\Omega}^{T}(R P^{2} R) \widetilde{\Omega} X.\nonumber
\end{align}
Since $P_{2\alpha}=P_{2\alpha}^T$, $P_{2\alpha}^{-2}=iE_{2\alpha}$, $R_{2\alpha}=R_{2\alpha}^{T}=R_{2\alpha}^{-1}=I_{\alpha}\oplus F_{\alpha}$, we get
\begin{align}\label{s2a}
(S_{2\alpha}^{-1})^{T}S_{2\alpha}^{-1}=& 
(R_{2\alpha}^{-1})^{T}(P_{2\alpha}^{-1})^{T}P_{2\alpha}^{-1}R_{2\alpha}^{-1}
=
R_{2\alpha}P_{2\alpha}^{-2}R_{2\alpha}
%= R_{2\alpha}^{-1} (i E_{2\alpha}) R_{2\alpha} 
%\nonumber \\
%\end{align*}
%
%
%The computation
%\begin{align*}
%R_{2\alpha}P_{2\alpha}^{2}R_{2\alpha}=
=(I_{\alpha} \oplus F_{\alpha})(iE_{2\alpha})(I_{\alpha} \oplus F_{\alpha})=   \nonumber\\
= &
i\begin{bmatrix}
0 & E_{\alpha}F_{\alpha}\\
F_{\alpha}E_{\alpha} & 0 
\end{bmatrix}
%=
%\begin{bmatrix}
%0 & F_{\alpha}\\
%(-1)^{\alpha-1}F_{\alpha} & 0 
%\end{bmatrix}\\
= 
i\big(E_{\alpha}F_{\alpha}\oplus (E_{\alpha}F_{\alpha})^{T}\big)
\begin{bmatrix}
0 & I_{\alpha}\\
I_{\alpha} & 0 
\end{bmatrix}.  
%=
%(-1)^{\alpha-1}
%\begin{bmatrix}
%0 & I_{\alpha}\\
%I_{\alpha} & 0 
%\end{bmatrix}    
%\big( F_{\alpha}\oplus (-1)^{\alpha-1} F_{\alpha}\big)
\end{align}
Thus we can decompose $(S^{-1})^{T}S^{-1} $ into a product of block diagonal matrices $iD$, $V$ with diagonal and anti-diagonal blocks, respectively; this is the trick of the proof:
%\vspace{-1mm}
%
\begin{equation}\label{deDV}
(S^{-1})^{T}S^{-1}
%RP^{2}R 
=\bigoplus_{r=1}^{N}\left(\bigoplus_{j=1}^{m_r} 
i
\Big(
\big(E_{\alpha_r}F_{\alpha_r}\oplus (E_{\alpha_r}F_{\alpha_r})^{T}\big)
\begin{bmatrix}
0 & I_{\alpha_r}\\
I_{\alpha_r} & 0 
\end{bmatrix}
\Big)\right)
=iDV,
\end{equation}
where 
$D:=\bigoplus_{r=1}^{N}\big(\bigoplus_{j=1}^{m_r}  (E_{\alpha_r}F_{\alpha_r}\oplus (E_{\alpha_r}F_{\alpha_r})^{T})  \big)$ and 
$V:=\bigoplus_{r=1}^{N}\big(\bigoplus_{j=1}^{m_r}  
\begin{bsmallmatrix}
0 & I_{\alpha_r}\\
I_{\alpha_r} & 0 
\end{bsmallmatrix}
 \big)$.

Furthermore, (\ref{deDV}) yield the following decomposition of $\widetilde{\Omega}^{T} \big((S^{-1})^{T}S^{-1}\big) \widetilde{\Omega}$:
\begin{align*}
&\widetilde{\Omega}^{T} \big((S^{-1})^{T}S^{-1}\big
) \widetilde{\Omega}=i(\widetilde{\Omega}^{T} D \widetilde{\Omega})(\widetilde{\Omega}^{T} V \widetilde{\Omega})=iBW,
\qquad n:=\sum_{r=1}^{N}\alpha_r m_r,\\
& W:=\widetilde{\Omega}^{T} V \widetilde{\Omega}=
\begin{bmatrix}
0 & I_{n}\\
I_{n} & 0 
\end{bmatrix},\quad
%E_2\Big(\bigoplus_{r=1}^{N}\big(\bigoplus_{j=1}^{m_r} I_{\alpha_r} \big)\Big),\\
B:=\widetilde{\Omega}^{T} D \widetilde{\Omega}=\bigoplus_{r=1}^{N}\big(\bigoplus_{j=1}^{m_r} E_{\alpha}F_{\alpha_r}  \big)\oplus \bigoplus_{r=1}^{N}\big(\bigoplus_{j=1}^{m_r} (E_{\alpha_r}F_{\alpha_r})^{T}  \big).
%.=
%\bigoplus_{r=1}^{N}\left(\bigoplus_{j=1}^{m_r} F_{\alpha}  \right)\oplus \bigoplus_{r=1}^{N}\left(\bigoplus_{j=1}^{m_r} (-1)^{\alpha-1} F_{\alpha_r}  \right);
\end{align*}
%
%note that $E_{\alpha}F_{\alpha}E_{\alpha}=(-1)^{\alpha-1}$.
%
Hence (\ref{QTQH}) further becomes
%
%\vspace{-1mm}
\begin{align}\label{QTQH2}
%I=   & \big(P^TR^{T}\widetilde{\Omega}X^T\widetilde{\Omega}^{T}(R^{-1})^{T}(P^{-1})^T\big)\big(P^{-1}R^{-1}\widetilde{\Omega}X \widetilde{\Omega}^{T}R P \big)\nonumber\\
%\widetilde{\Omega}^{T}(R^T)^{-1}(P^T)^{-1}P^{-1}R^{-1}\widetilde{\Omega} =   &  X^T (\widetilde{\Omega}^{T}R^{-1})^{T}(P^{-1})^{2} R \widetilde{\Omega} X\nonumber\\
 %I = & S_{\epsilon}^2\overline{P}^2Y^T(P^{-1})^2S_{\epsilon}^{2}Y\nonumber\\
% I = & S_{\epsilon}^2 P^2X^TP^{-2}S_{\epsilon}^{2}X\nonumber\\
%  S_{\epsilon}^2 = & \big(P^2 X P^{2}\big)^{T}S_{\epsilon}^{2}X\nonumber\\
%\widetilde{\Omega}^{T}(R^T)^{-1}P^{2} R^{-1}\widetilde{\Omega} = &  X^{T} \widetilde{\Omega}^{T}(R^{-1})^{T} P^{2} R \widetilde{\Omega} X.\nonumber
iBW = &  X^{T} (i BW) X \nonumber\\
B = & X^{T} B (W X W^{-1})\nonumber\\
B_1 \oplus B_1^{T} = & (X_1^{T} \oplus X_2^{T}) (B_1 \oplus B_1^{T}) (W (X_1 \oplus X_2) W^{-1})\\
B_1 \oplus B_1^{T} = & (X_1^{T} \oplus X_2^{T}) (B_1 \oplus B_1^{T}) (X_2 \oplus X_1)\nonumber\\
B_1 = &  X^{T}_1 B_1 X_2,\qquad \quad (B_1^{T}=X^{T}_2 B_1^{T} X_1), \nonumber
\end{align}
where $B=B_1 \oplus B_1^{T}$ with $B_1:=\bigoplus_{r=1}^{N}\big(\bigoplus_{j=1}^{m_r} E_{\alpha}F_{\alpha}  \big)$.

By conjugating both hand sides of the last equation of (\ref{QTQH2}) by $\Omega$ from Lemma \ref{lemaP} and then slightly manipulating it,
we get
\begin{align}\label{ortoD3}
 %\Omega^{T}S^{T}\mathcal{I}S\Omega = & \Omega^{T} Y^{T}\Omega\Omega^{T}S^{T}\mathcal{I}S \Omega\Omega^{T} Y\Omega\\
  \Omega^{T}B_1\Omega = & (\Omega^{T} X_1\Omega)^{T}(\Omega^{T}B_1\Omega)(\Omega^{T} X_2\Omega)\\
\mathcal{B}_1= &  \mathcal{X}^{T}_1\mathcal{B}_1 \mathcal{X}_2,\nonumber
%\mathcal{B}_1\oplus \mathcal{B}_1^{T}= & (\mathcal{X}^{T}_1 \oplus \mathcal{X}^{T}_2 ) (\mathcal{B}_1\oplus \mathcal{B}_1^{T})(\mathcal{U}(\mathcal{X}^{}_1 \oplus \mathcal{X}^{}_2)\mathcal{U}),\\
%\mathcal{B}_1\oplus \mathcal{B}_1^{T}= & (\mathcal{X}^{T}_1 \oplus \mathcal{X}^{T}_2 ) (\mathcal{B}_1\oplus \mathcal{B}_1^{T})(\mathcal{X}^{}_2 \oplus \mathcal{X}^{}_1),\nonumber\\
% \mathcal{B}_1= &\mathcal{X}^{T}_1\mathcal{B}_1\mathcal{X}_2,\qquad  \mathcal{B}_1^{T}=\mathcal{X}^{T}_2\mathcal{B}_1^{T}\mathcal{X}_1, \nonumber
\end{align}
where $\mathcal{X}_1:=\Omega^{T} X_1\Omega$, $\mathcal{X}_2:=\Omega^{T} X_2\Omega$ with $\mathcal{X}_1, \mathcal{X}_2\in \mathbb{T}^{\alpha,\mu}$ (cf. Proposition \ref{lemanilpo}), and $\mathcal{B}_1:= \bigoplus_{r=1}^{N}\big( E_{\alpha_r}(I_{m_r})F_{\alpha_r}(I_{m_r})\big)$, in which $E_{\alpha_r}(I_{m_r})$ is a block anti-diagonal matrix with $I_{m_r}$ on the anti-diagonal (see (\ref{BBF})), and $ F_{\alpha_r}(I_{m_r}):=\bigoplus_{j=1}^{\alpha_r}(-1)^{j-1}I_{m_r}$.
%(with $-I_{m_r}$, $I_{m_r}$ alternating on the diagonal). 
Thus, the solution of (\ref{HQQH}) for $\mathcal{A}=\mathcal{K}$ as in (\ref{K1l}) is
\begin{align}\label{QPsi}
&Q=S^{-1}\widetilde{\Omega}(X_1\oplus X_2)\widetilde{\Omega}^{T}S
=\Psi^{-1} \big(\mathcal{X}_1\oplus (\mathcal{X}^{T}_1)^{-1}\big)\Psi, \qquad \mathcal{X}_1\in \mathbb{T}^{\alpha,\mu}\\
&\qquad \Psi:=(I\oplus \mathcal{B})\big((\Omega^{T}\oplus \Omega^{T})\widetilde{\Omega}^{T}S\big).\nonumber
\end{align}

Next, let
\begin{equation*}
%\label{K1l}
\mathcal{O}^{}
%=\bigoplus_{r=1}^{N} \mathcal{H}_{\alpha_r}(\lambda)
=\bigoplus_{r=1}^{N}\big( \bigoplus_{j=1}^{m_r}  e^{K_{\alpha_r}(\lambda)}\big), \qquad \lambda \in \mathbb{C}\setminus \{\pi k\}_{k\in \mathbb{T}},
%\mathcal{H}^{}
%=\bigoplus_{r=1}^{N} \mathcal{H}_{\alpha_r}(\lambda)
\end{equation*}
where $K_{\alpha_r}(\lambda)$ is as in (\ref{Km}) for $m=\alpha_r$.
%where $\mathcal{H}_{\alpha_r}(\lambda)$, $r\in \{1,\ldots,N\}$ is a %direct sum of all blocks of size $\alpha_r\times \alpha_r $ %corresponding to the eigenvalue $\lambda\geq 0$. 
Since $K_{\alpha_r}(\lambda)=S_{2\alpha_r}^{-1}\big(J_{\alpha_r}(\lambda)\oplus J_{\alpha_r}(_\lambda)\big)S_{2\alpha_r}$ and $e^{K_{\alpha_r}(\lambda)}=S_{2\alpha_r}^{-1}e^{J_{\alpha_r}(\lambda)\oplus J_{\alpha_r}(-\lambda)}S_{2\alpha_r}$, we obtain that $\mathcal{O}^{}$ is similar to

\begin{align*}%\label{JCF1o}
%&\mathcal{K}=P^{-1}R^{-1}\mathcal{J}R P,
%\qquad
\mathcal{J'}=\bigoplus_{r=1}^{N}\Big( \bigoplus_{j=1}^{m_r}\big(
% S_{2\alpha_r}^{-1} 
e^{ J_{\alpha_r}(\lambda)\oplus  J_{\alpha_r}(-\lambda)}
%S_{2\alpha_r}
 \big)\Big),\qquad \qquad (\mathcal{O}=S^{-1}\mathcal{J}'S),%\qquad\\e
\end{align*}
with the transition matrix $S$ from (\ref{SRP}). Using $\widetilde{\Omega}$ from (\ref{J2F}) we further get
\begin{equation*}%\label{J2F}
\widetilde{\mathcal{J}}':=\widetilde{\Omega}^{T}\mathcal{J}'\widetilde{\Omega}=
\bigoplus_{r=1}^{N}\Big( \bigoplus_{j=1}^{m_r} e^{ J_{\alpha_r}(\lambda)}\Big)\oplus 
\bigoplus_{r=1}^{N}\Big( \bigoplus_{j=1}^{m_r}  e^{J_{\alpha_r}(-\lambda)}\Big). 
%\qquad 
%\widetilde{\Omega}=\widetilde{\Omega}_1\widetilde{\Omega}_2.
\end{equation*}
Then the equation (\ref{HQQH}) for $\mathcal{A}=\mathcal{O}$ transforms to (\ref{JQQJ}) for $\mathcal{J}=\widetilde{\mathcal{J}}'$:
%\vspace{-2mm}
%
\begin{equation*}%\label{wJQQwJo}
\widetilde{\mathcal{J}}'X=X\widetilde{\mathcal{J}}', \qquad X=(\widetilde{\Omega}^{T} S) Q (S^{-1}\widetilde{\Omega}), \qquad \qquad (Q \textrm{ orthogonal}).
\end{equation*}
Lemma \ref{resAoXXAo} implies that the solution $X$ of this equation coincides with the solution of (\ref{wJQQwJ}). Furthermore, orthogonality of $Q$ leads to (\ref{QTQH}) and eventually we obtain (\ref{QPsi}). It concludes the proof of Theorem \ref{stabw} (\ref{stabw1}).
\end{proof}

\begin{proof}[Proof of Corrolary \ref{stabr}]
Due to Proposition \ref{stabs11} (\ref{stabs11a}), (\ref{stabs11ao}) it suffices to consider the case
\small
\begin{equation}\label{KOr}
\mathcal{K}=\bigoplus_{j=1}^{m}
\begin{bmatrix}
0 & \sigma\\
-\sigma & 0 
\end{bmatrix},
\quad \sigma\in \mathbb{R}\setminus \{0\}, \qquad
%\big(
\mathcal{O}=
%e^{i\mathcal{K}}=
\bigoplus_{j=1}^{m}
\begin{bmatrix}
\cos (\sigma) & \sin (\sigma)\\
-\sin (\sigma) & \cos (\sigma)
\end{bmatrix},
\quad \sigma \neq \{2k\pi\}_{k\in \mathbb{Z}}.
%e^{i\sigma}\neq 1.
%\big).
%\sigma \in \mathbb{R}^{+}\setminus \{k\pi\}_{k\in \mathbb{Z}}).
\end{equation}
\normalsize
%
%with the corresponding Jordan canonical form $J=\lambda \mathcal{K}=\bigoplus_{j=1}^{m_r}(I_{m_r}\oplus -I_{m_r})$ and the transition matrix is
%
We apply the proof of Theorem \ref{stabw} for a simple case $N=\alpha_1=1$, $m_1=m$, $\lambda=i\sigma$ with $\sigma \in \mathbb{R}\setminus \{0\}$ ($e^{i\sigma}\neq 1$). Now, $\widetilde{\Omega}_2=I_{2m}$ in (\ref{perK2}), hence $\widetilde{\Omega}=\Omega_{2,m}$ in (\ref{J2F}). Furthermore, from (\ref{QPsi}) we deduce the solution $Q$ of the equation (\ref{HQQH}) for $\mathcal{A}=\mathcal{K}$ and $\mathcal{A}=\mathcal{O}$ (from (\ref{KOr})):
\begin{equation}\label{QX1X2S}
Q=S^{-1}\Omega_{2,m}(X_1\oplus (X_1^{-1})^{T}) \Omega_{2,m}^{T}  S, \quad
X_1\in \GL_{m}(\mathbb{C}), \,\,\,\,\,
S= \bigoplus_{j=1}^{m}
\tfrac{1}{\sqrt{2}}\begin{bsmallmatrix}
1 & i\\
-i & -1 
\end{bsmallmatrix}.
\end{equation}
Observe further that
%
%\small
\begin{align*}
&\Omega_{2,m}^{T}Q \Omega_{2,m}= (\Omega_{2,m}^{T} S^{-1}\Omega_{2,m}) (X_1\oplus X_2) (\Omega_{2,m}^{T} S \Omega_{2,m})=\\
%= & ((\widetilde{\Omega})^{T}P^{-1} \widetilde{\Omega})((\widetilde{\Omega})^{T}R^{-1} \widetilde{\Omega})((\widetilde{\Omega})^{T} \widetilde{X}\widetilde{\Omega})(\widetilde{\Omega})^{T} R \widetilde{\Omega})(\widetilde{\Omega})^{T} P \widetilde{\Omega})\\
&=
%\frac{1}{\sqrt{2}}
\frac{1}{2}\begin{bmatrix}
I_{m} & iI_{m} \\
-iI_{m} & -I_{m} 
\end{bmatrix}
%(I_{m_1}\oplus -I_{m_1})
(X_1\oplus (X_1^{-1})^{T})
%(I_{m_1}\oplus -I_{m_1})
%\frac{1}{\sqrt{2}}
\begin{bmatrix}
I_{m} & iI_{m} \\
-iI_{m} & -I_{m} 
\end{bmatrix}
%\\
=  \frac{1}{2}\begin{bmatrix}
X_1+(X_1^{-1})^{T} & i(X_1-(X_1^{-1})^{T}) \\
-i(X_1-(X_1^{-1})^{T})  & X_1+(X_1^{-1})^{T}
\end{bmatrix}.
\end{align*} 
\normalsize
It follows that $Q$ is real if and only if $X_1-(X_1^{-1})^{T}$ is purely imaginary and $X_1+(X_1^{-1})^{T}$ is real, which holds precisely when $X_1^{-1}=\overline{X}_1^{T}$.
%(complex conjugate).
It concludes the proof.
%Combining this with $X_1=(X_2^{-1})^{T}$ in (\ref{QX1X2S}), yields $\overline{X}_1^{T}X_1=I$, and 
\end{proof}

\begin{proof}[Proof of Theorem \ref{stabw} (\ref{stabw2})]
Let $K_{\alpha_r}(0)$, $L_{\alpha_r}$ (with $\alpha_r$ odd) be as in (\ref{Km}), (\ref{Lm}), and let
\begin{equation}\label{Kcase2}
\mathcal{K}^{} 
%=\bigoplus_{r=1}^{N}\big( \bigoplus_{j=1}^{m_r} (K_{2\alpha_r}(0)\oplus K_{2\alpha_r}(0))\big)\oplus \bigoplus_{r=1}^{N'}\big( \bigoplus_{j=1}^{m_r'} (L_{2\alpha_r'-1}\big),
=\bigoplus_{r=1}^{N}\big( \bigoplus_{j=1}^{m_r} \widetilde{K}_r \big),
\qquad \widetilde{K}_{r}:=
\left\{
\begin{array}{ll}
\hspace{-1mm}K_{\alpha_r}(0),   &  \hspace{-1mm} 
%\textrm{all } 
\alpha_r \textrm{ even}\\
\hspace{-1mm}L_{\alpha_r},             &   \hspace{-1mm} 
%\textrm{all } 
\alpha_r \textrm{ odd}
\end{array}
\right..
\end{equation}
%
%\vspace{-1mm}
The Jordan canonical form of $\mathcal{K}$ is 
\begin{equation}\label{Jcase2}
\mathcal{J}
=\bigoplus_{r=1}^{N}\Big( \bigoplus_{j=1}^{m_r} \widetilde{J}_{r} \Big),
\quad \widetilde{J}_{r}=
\left\{
\begin{array}{ll}
\hspace{-1mm}J_{\alpha_r}(0) \oplus J_{\alpha_r}(0),   &  \hspace{-1mm} 
%\textrm{all } 
\alpha_r \textrm{ even}\\
\hspace{-1mm}J_{\alpha_r}(0),             &   \hspace{-1mm} 
%\textrm{all } 
\alpha_r \textrm{ odd}
\end{array}
\right.,
\qquad (\mathcal{K}=S^{-1}\mathcal{J}S),
\end{equation}
where the corresponding transition matrix is:
\begin{align}\label{SRP2}
& S=
\bigoplus_{r=1}^{N}\Big( \bigoplus_{j=1}^{m_r}\widetilde{S}_r
% \big(R_{\alpha_r}P_{\alpha_r}\big)
\Big), \qquad
\widetilde{S}_r
:=
\left\{
\begin{array}{ll}
R_{2\alpha_r}P_{2\alpha_r}, &  \alpha_r \textrm{ even}\\
R_{\alpha_r}P_{\alpha_r}, &  \alpha_r \textrm{ odd}
\end{array}
\right.,
\end{align}
in which $R_{\alpha}:=
\left\{
\begin{array}{ll}
I_{\alpha} \oplus F_{\alpha}, &  \alpha \textrm{ even}\\
I_{(\alpha+1)/2} \oplus F_{(\alpha-1)/2}, &  \alpha \textrm{ odd}
\end{array}
\right.$, and $P_{\alpha}$, $F_{\alpha}$ are defined by (\ref{SRP}).

The equation (\ref{HQQH}) for $\mathcal{A}=\mathcal{K}$ as in (\ref{Kcase2}) transforms to (\ref{JQQJ}) for $\mathcal{J}$ as in (\ref{Jcase2}) and $U=S$ (cf. (\ref{SRP2})):
\begin{equation}\label{JXXJS}
\mathcal{J}X=X\mathcal{J}, \qquad X=SQS^{-1} \qquad \quad (Q \textrm{ orthogonal}).
\end{equation}
Proposition \ref{resAoXXA} (\ref{resAoXXA1}) yields the solution of this equation $X=[X_{rs}]_{r,s=1}^{N}$ with an $\widetilde{m}_r$-by-$\widetilde{m}_s$ block matrix $X_{rs}$ whose blocks are $\alpha_r$-by-$\alpha_s$ rectangular Toeplitz
of the form (\ref{QTY}), where 
\small
$
\widetilde{m}_r:=\left\{
\begin{array}{ll}
2m_r,  &  \alpha_r \textrm{ even} \\
m_r, &  \alpha \textrm{ odd}
\end{array}
\right.
$.
\normalsize
Furthermore, the orthogonality of $Q$ is characterized by the equation (\ref{QK}) for $U=S$ (see (\ref{SRP2})):
%
%\vspace{-1mm}
\begin{align}\label{QTQcase2}
%I=   & \big(S^{T}X^T(S^{-1})^T\big)\big(S^{-1}X S \big)\\
(S^T)^{-1}S^{-1} =   &  X^T \big((S^T)^{-1}S^{-1}\big) X.
%\nonumber
 %I = & S_{\epsilon}^2\overline{P}^2Y^T(P^{-1})^2S_{\epsilon}^{2}Y\nonumber\\
% I = & S_{\epsilon}^2 P^2X^TP^{-2}S_{\epsilon}^{2}X\nonumber\\
%  S_{\epsilon}^2 = & \big(P^2 X P^{2}\big)^{T}S_{\epsilon}^{2}X\nonumber\\
%(R^T)^{-1}G R^{-1} = &  X^{T} (R^{-1})^{T} G RX.\nonumber\\
%E(R^T)^{-1}G R^{-1} = &  E  X^{T}E (E(R^{-1})^{T} G R) X.\nonumber\\
%B = &  E  X^{T} E B X.\nonumber
%I = &  E  X^{T} EU XU^{-1}.\nonumber
\end{align}

We now apply the same trick as in the proof of Theorem \ref{stabw} (\ref{stabw1}) and write $(S^{-1})^{T}S^{-1}$ as a product of block diagonal matrices 
%$iE$ and $B$ 
one with diagonal and one with anti-diagonal blocks. First, set $S_{\alpha}:=R_{\alpha}P_{\alpha}$. Using the computation (\ref{s2a}) we then get
\begin{align}\label{sodS}
(S_{2\beta}^{-1})^{T}S_{2\beta}^{-1}
=& 
%(R_{4\alpha}^{-1})^{T}(P_{4\alpha}^{-1})^{T}P_{4\alpha}^{-1}R_{4\alpha}^{-1}
%=
%R_{4\alpha}^{-1}P_{4\alpha}^{-2}R_{4\alpha}
%=i(I_{2\alpha}\oplus I_{2\alpha}) E_{2\alpha} (I_{2\alpha}\oplus I_{2\alpha})\\
%=& 
%i(I_{\alpha} \oplus F_{\alpha})E_{2\alpha}(I_{\alpha} \oplus F_{\alpha})
%=
i\begin{bmatrix}
0 & E_{\beta}F_{\beta}\\
F_{\beta}E_{\beta} & 0 
\end{bmatrix}
= 
%i
%(E_{\alpha} \oplus E_{\alpha})
%\begin{bmatrix}
%0 & F_{\alpha}\\
%E_{\alpha} F_{\alpha} E_{\alpha}& 0 
%\end{bmatrix}
%=
i
(E_{\beta} \oplus E_{\beta})
\begin{bmatrix}
0 & F_{\beta}\\
- F_{\beta} & 0 
\end{bmatrix},
\end{align}
and similarly by applying $P_{\alpha}=P_{\alpha}^T$, 
$P_{\alpha}^{-2}=iE_{\alpha}$, 
$R_{\alpha}=R_{\alpha}^{T}=R_{\alpha}^{-1}$ we obtain:
\begin{align}\label{lihS}
(S_{2\beta-1}^{-1})^{T}S_{2\beta-1}^{-1}=& 
(R_{2\beta-1}^{-1})^{T}(P_{2\beta-1}^{-1})^{T}P_{2\beta-1}^{-1}R_{2\beta-1}^{-1}
=
R_{2\beta-1}P_{2\beta-1}^{-2}R_{2\beta-1}
%= R_{2\alpha-1}^{-1} (i E_{2\alpha-1}) R_{2\alpha-1}
\nonumber\\
= & 
(I_{\beta} \oplus F_{\beta-1})(iE_{2\beta-1})(I_{\beta} \oplus F_{\beta-1})
=
i \begin{bmatrix}
0 & 0 & E_{\beta-1} F_{\beta-1}\\
0 & 1 & 0\\
F_{\beta-1}E_{\beta-1} & 0 & 0
\end{bmatrix}\\
= &
i \begin{bmatrix}
0 & 0 & E_{\beta-1}\\
0 & 1 & 0\\
E_{\beta-1} & 0 & 0
\end{bmatrix}
(E_{\beta-1} F_{\beta-1}E_{\beta-1} \oplus 1\oplus F_{\beta-1})
%= 
%E_{2\alpha-1}  \big((-1)^{\alpha} F_{\alpha-1} \oplus - F_{\alpha}\big)
= 
i E_{2\beta-1} \big((-1)^{\beta} F_{2\beta-1}\big).\nonumber
%= 
%i\big(E_{\alpha}F_{\alpha}\oplus (E_{\alpha}F_{\alpha})^{T}\big)
%\begin{bmatrix}
%0 & I_{\alpha}\\
%I_{\alpha} & 0 
%\end{bmatrix}.  
\end{align}
From (\ref{sodS}), (\ref{lihS}) it follows that
%\[
%(\widetilde{S}_r^{-1})^{T}\widetilde{S}_r^{-1}
%=
%\left\{
%\begin{array}{ll}
%(S_{2\alpha_r}^{-1})^{T}S_{2\alpha_r}^{-1}, &  \alpha_r \textrm{ even}\\
%(S_{\alpha_r}^{-1})^{T}S_{\alpha_r}^{-1}, &  \alpha_r \textrm{ odd}
%\end{array}
%\right.
%=
%\left\{
%\begin{array}{ll}
%\begin{bsmallmatrix}
%0 & F_{\alpha_r}\\
%- F_{\alpha_r} & 0 
%\end{bsmallmatrix}, &  \alpha_r \textrm{ even}\\
%(-1)^{(\alpha_r+1)/2} F_{\alpha}, &  \alpha_r \textrm{ odd}
%\end{array}
%\right..
%\]
%
\begin{align*}
(S^{-1})^{T} & S^{-1}  =i EB,\\
%\end{align*}
%\begin{align*}
%&
E:=\bigoplus_{r=1}^{N}\big(\bigoplus_{j=1}^{\widetilde{m}_r}  E_{\alpha_r}\big), \qquad
B:=\bigoplus_{r=1}^{N}\big(\bigoplus_{j=1}^{m_r} 
%F_{4\alpha_r}
\widetilde{B}_r \big),&
\quad
\widetilde{B}_r:=\left\{
\begin{array}{ll}
\begin{bsmallmatrix}
0 & F_{\alpha_r}\\
- F_{\alpha_r} & 0 
\end{bsmallmatrix}, &  \alpha_r \textrm{ even}\\
(-1)^{(\alpha_r+1)/2} F_{\alpha_r}, &  \alpha_r \textrm{ odd}
\end{array}
\right..
\end{align*}
The equation (\ref{QTQcase2}) thus transforms to
%
%\vspace{-1mm}
\begin{align}\label{QTQ2c}
%I=   & \big(P^TR^{T}X^T(R^{-1})^{T}(P^{-1})^T\big)\big(P^{-1}R^{-1}X R P \big)\nonumber\\
%(R^T)^{-1}(P^T)^{-1}P^{-1}R^{-1} =   &  X^T (R^{-1})^{T}(P^{-1})^{2} R X\nonumber\\
 %I = & S_{\epsilon}^2\overline{P}^2Y^T(P^{-1})^2S_{\epsilon}^{2}Y\nonumber\\
% I = & S_{\epsilon}^2 P^2X^TP^{-2}S_{\epsilon}^{2}X\nonumber\\
%  S_{\epsilon}^2 = & \big(P^2 X P^{2}\big)^{T}S_{\epsilon}^{2}X\nonumber\\
%(R^T)^{-1}G R^{-1} = &  X^{T} (R^{-1})^{T} G RX.\nonumber\\
iE B = &  E  X^{T}(iE B) X\\
B = & ( E  X^{T} E ) B X.\nonumber
%I = &  E  X^{T} EU XU^{-1}.\nonumber
%D V = &  E  Y^{T} E D V X.\nonumber\\
%D = & ( E  Y^{T} E) D V X V.\nonumber
\end{align}
%\vspace{-2mm}
%
%
We conjugate (\ref{QTQ2c}) with $\Omega=\bigoplus_{r=1}^{N}\Omega_{\alpha_r,\widetilde{m}_r}$ (see Lemma \ref{lemaP}) and manipulate it:
%
%\vspace{-1mm}
\begin{align}\label{ortoD2}
 \Omega^{T} B\Omega = & \big(\Omega^{T}E\Omega \big)\big(\Omega^{T} X\Omega\big)^{T}\big(\Omega^{T}E\Omega\big)\big(\Omega^{T}B \Omega\big)\big(\Omega^{T} X\Omega\big)\\
  \mathcal{B} = &\mathcal{F}\mathcal{X}^{T}\mathcal{F}\mathcal{B}\mathcal{X}; \nonumber
\end{align}
where $\mathcal{F}:= \Omega^{T}E\Omega=\bigoplus_{r=1}^{N} E_{\alpha_r}(I_{\widetilde{m}_r})$, $\mathcal{X}=\Omega^{T}X\Omega$ is of the form (\ref{0T0}) for $\alpha=(\alpha_1,\ldots,\alpha_N)$, $\mu=\widetilde{\mu}=:(\widetilde{m}_1,\ldots,\widetilde{m}_N)$, and we have set
\begin{align*}
\mathcal{B}:= & \Omega^{T}B\Omega=\bigoplus_{r=1}^{N} 
\big(
\bigoplus_{j=1}^{\alpha_r}
(-1)^{j-1}B_r
\big),\qquad
B_r:=\left\{
\begin{array}{ll}
\bigoplus_{j=1}^{m_r}
\begin{bsmallmatrix}
0 & -1\\
1 & 0 
\end{bsmallmatrix}, &  \alpha_r \textrm{ even}\\
(-1)^{(\alpha_r-1)/2} I_{\alpha_r}, &  \alpha_r \textrm{ odd}
\end{array}
\right..
\end{align*}
The solution of (\ref{HQQH}) for $\mathcal{A}=\mathcal{K}$ (see (\ref{Kcase2})) is thus $Q=\Psi^{-1}\mathcal{X}\Psi$ for $\Psi:=\Omega^{T}S$ (cf. (\ref{JXXJS})), with $\mathcal{X}$ a solution of (\ref{ortoD2}), which is obtained by applying  Lemma \ref{EqT}.

We proceed with
\begin{equation}\label{Kcase2o}
\mathcal{O}^{} 
=\varepsilon \bigoplus_{r=1}^{N}\big( \bigoplus_{j=1}^{m_r} e^{\widetilde{K}_r} \big),\qquad \varepsilon\in\{-1,1\},
%\quad\qquad 
%\widetilde{K}_{r}=
%\left\{
%\begin{array}{ll}
%\hspace{-1mm}K_{\alpha_r}(0),   &  \hspace{-1mm} 
%\alpha_r \textrm{ even}\\
%\hspace{-1mm}L_{\alpha_r},             &   \hspace{-1mm} 
%\alpha_r \textrm{ odd}
%\end{array}
%\right.;
\end{equation}
%\footnote{}
recall that $K_{\alpha_r}(0)$, $L_{\alpha_r}$, $\widetilde{K}_{r}$ are defined in (\ref{Km}), (\ref{Lm}), (\ref{Kcase2}). It is similar to 
\begin{equation}\label{Jcase2o}
\mathcal{J}'
=\bigoplus_{r=1}^{N}\Big( \bigoplus_{j=1}^{m_r} e^{\widetilde{J}_{r}} \Big),
\quad \widetilde{J}_{r}:=
\left\{
\begin{array}{ll}
\hspace{-1mm}J_{\alpha_r}(0) \oplus J_{\alpha_r}(0),   &  \hspace{-1mm} 
%\textrm{all } 
\alpha_r \textrm{ even}\\
\hspace{-1mm}J_{\alpha_r}(0),             &   \hspace{-1mm} 
\alpha_r \textrm{ odd}
\end{array}
\right.
\qquad (\mathcal{O}=S^{-1}\mathcal{J}'S),
\end{equation}
in which the corresponding transition matrix $S$ is as in (\ref{SRP2}). The equation (\ref{HQQH}) for $\mathcal{A}=\mathcal{O}$ as in (\ref{Kcase2o}) transforms to (\ref{JQQJ}) for $\mathcal{J}=\widetilde{\mathcal{J}}'$ as in (\ref{Jcase2o}):
% of the form (\ref{JCF1o})
%\vspace{-2mm}
%
\begin{equation}\label{wJQQwJo2}
\widetilde{\mathcal{J}}'X=X\widetilde{\mathcal{J}}', \qquad X=S Q S^{-1}, \qquad \qquad (Q \textrm{ orthogonal}).
\end{equation}
Applying Lemma \ref{resAoXXAo} to (\ref{wJQQwJo2}) gives the same solution $X$ as obtained for the equation (\ref{JXXJS}), while orthogonality of $Q$ yields the equation (\ref{QTQcase2}). The rest of the proof works mutatis mutandis as above for (\ref{HQQH}) with $\mathcal{A}=\mathcal{K}$.
\end{proof}
%
%
%
%

%\bigskip
%\noindent
%\textbf{Acknowledgement.}\\
%The research was supported by 
%Slovenian Research Agency (grants no. P1-0291 and no. J1-3005).

%\vspace{-2mm}

%\bibliographystyle{amsplain}

\end{document}